\documentclass[11pt,english]{article}
\usepackage{babel,amsfonts}
\usepackage{indentfirst,graphics,graphicx}
\usepackage{a4wide}

\usepackage[colorlinks=true,breaklinks=true,linkcolor=blue]{hyperref}

\newtheorem{theorem}{Theorem}
\newtheorem{definition}[theorem]{Definition}
\newtheorem{lemma}[theorem]{Lemma}

\newtheorem{remark}[theorem]{Remark}

\newcommand{\cqfd}{\mbox{}\nolinebreak\hfill\rule{2mm}{2mm}\medskip\par}

\newenvironment{proof}[1] {\par\noindent{\bf Proof. }{#1}}{\cqfd}

\def\Om     {\Omega}
\def\eps     {\varepsilon}
\def\div      {\nabla\cdot}
\def\grad    {\nabla }
\def\jnt        {\displaystyle\int}
\def\fra#1#2{\displaystyle\frac{\mathstrut #1}{\mathstrut #2}}
\def\p        {\psi }
\def\u        {\textit{\textbf{u}}}
\def\w        {\textit{\textbf{w}}}
\def\v        {\textit{\textbf{v}}}
\def\x        {\textit{\textbf{x}}}

\def\n        {\textit{\textbf{n}}}
\def\H        {\textit{\textbf{H}}}
\def\L        {\textit{\textbf{L}}}
\def\V        {\textit{\textbf{V}}}

\begin{document}
\title{Convergence to equilibrium of global weak solutions for a Cahn-Hilliard-Navier-Stokes vesicle model
}
\author{\sc
Blanca~Climent-Ezquerra, Francisco~Guill\'en-Gonz\'alez
\\ \\
Dpto.~Ecuaciones Diferenciales y An\'alisis Num\'erico,
Universidad de Sevilla, \\ Aptdo.~1160, 41080 Sevilla, Spain.\\
E-mails: bcliment@us.es, guillen@us.es}
\maketitle
\begin{abstract}
In this paper, we introduce a model describing the dynamic of vesicle membranes within an incompressible viscous fluid in $3D$ domains. The system consists of the Navier-Stokes equations, with an extra stress tensor depending on the membrane, coupled with a Cahn-Hilliard phase-field equation associated to a bending energy plus a penalization term  related to the area conservation.  This problem has a dissipative in time free-energy which leads, in particular, to prove the existence of global in time weak solutions. We analyze the large-time behavior of the weak  solutions. 
By using a modified Lojasiewicz-Simon's result, we prove the convergence as time goes to infinity of each (whole) trajectory to a single equilibrium. 
Finally, the  convergence of the trajectory of the phase is improved by imposing more regularity on the domain and initial phase. 
\end{abstract}


\section{Introduction}  \label{intro}
A type of differential equations modeling the dynamic of  vesicle membranes within  an incompressible viscous fluids was introduced by Helfrich \cite{Helfrich}. The membranes are composed by lipid bilayers, which under appropriate conditions, withdraw into themselves forming a sort of bag, named vesicle. The static equilibrium configurations of the vesicle membranes can be obtained minimizing the Helfrich bending elastic energy, under the constraints of fixed  area and volume.

 Numerous studies have been devoted to this type of models and a detailed description of them can be seen in  \cite{DuLiLiu}  and references therein. A phase function is also  used in subsequent papers to model  vesicle membranes as  diffuse interfaces.  In  \cite{LiuTakahashiTucsnak} and \cite{WuXu},  a coupled  Allen-Cahn and Navier-Stokes problem is studied approaching both constrains, area and volume, via a penalization functional. 
 On the other hand, without taking into account the vesicle-fluid interaction, a Cahn-Hilliard phase-field model for vesicle membranes  is introduced in \cite{CampHdezMach},  and the well-posedness of a phase-field approximation satisfying area and volume constraints pointwisely is established in \cite{Colli-Laurencot}. 
 
 In this paper, a Cahn-Hilliard-Navier-Stokes model will be  considered. Since the volume constraint is implicitly satisfied for the Cahn-Hilliard equation, only  the surface area constraint will be approximated via penalization. Moreover, the resulting problem will be thermodynamically consistent because there exists a  free-energy (kinetic plus bending plus penalized one) dissipative in time along the trajectories. This fact is used  to prove the existence of global in time weak solutions. 

On the other hand, the large-time behavior of the solutions will be analyzed following the way of \cite{GalGrasselli}, \cite{PRS}, \cite{nu}. Firstly, we prove that the $\omega$-limit set  for weak solutions is composed by  critical points of the free-energy. After that, by using a modified Lojasiewicz-Simon's result we demonstrate the convergence of the whole trajectory to a single equilibrium. Finally, the convergence is improved  for the phase assuming more regular data, but without using strong regularity for the velocity and pressure variables.

The main novelties  
in this paper are the following:
\begin{itemize}
\item
The introduction of a  new Navier-Stokes-Cahn-Hilliard problem modeling vesicle-fluid  interactions.
\item
The proof of a modified Lojasiewicz-Simon's result associated to weak solutions of a fourth order elliptic problem, which is adapted to the exigences of this new model. In fact, it was not possible to apply any already known Lojasiewicz-Simon's result to this problem.
\item
 Given a global weak solution,  we choose a special regularized energy  
  satisfying the energy's law  inequality either in a integral version for all time interval  and in a differential version  a.e.~in time (see (\ref{energyeqint-b}) and (\ref{energy-eqs}) below).
\item
The convergence of each  trajectory of  weak solutions to a unique equilibrium point, imposing that the regularized energy coincides a.e.~in time with the energy evaluated in the weak solution. 
\end{itemize}


The current paper is organized as follows. We explain the model in Section~\ref{TheModel} and give some preliminary results in Section~\ref{sec:prelim}. 
Section~\ref{LSR} is devoted to state and prove a new  Lojasiewicz-Simon's result. 
In Section~\ref{WS} we obtain, via a Galerkin method, the existence of global in time weak solutions satisfying an integral energy's law a.e.~time interval. Moreover, we found a function defined for all time $t$ equal to the free-energy a.e.\ that satisfies the integral energy's law for all time and also a differential version of the energy inequality a.e.~time.
Section~\ref{sec:time-inftys} is devoted to the study of convergence at infinite time for global weak solutions. In fact, we prove that the $\omega$-limit set consists of  critical points. After that, by using the modified  Lojasiewicz-Simon's result, we demonstrate that each trajectory converges to a single equilibrium. In Section~\ref{HS}, some global in time strong  estimates  are obtained for the phase variable, which allow to  improve  the norm where  the phase trajectory converges.
\section*{Notations}
\begin{itemize}
\item
In general, the notation will be abridged. We set $L^p=L^p(\Om)$,
$p\geq 1$, $H^1_0=H^1_0(\Om)$, etc. If $X=X(\Om)$ is a space of
functions defined in the open set $\Om$, we denote by $L^p(0,T;X)$ the
Banach space $L^p(0,T;X(\Omega))$. Also, boldface letters will be used for
vectorial spaces, for instance ${\textbf{L}}^2=L^2(\Omega)^N$.
\item
The $L^p$-norm is denoted by $\vert\cdot\vert_p$, $1\leq
p\leq\infty$, the $H^m$-norm  by $\Vert\cdot\Vert_m$ (in particular
$\vert\cdot\vert_2=\Vert\cdot\Vert_0$). The inner product
of $L^2(\Om)$ is denoted by $(\cdot,\cdot)$. 
\item
We set ${\mathcal V}$ the space formed by all fields $\u\in
C_0^\infty(\Om)^N$ satisfying $\div \u=0$. We denote $\H$
(respectively \V) the closure of ${\mathcal V} $ in $\L^2$
(respectively $\H^1$). $\H$ and $\V$ are Hilbert spaces for the
norms $\vert\cdot\vert_2$ and $\Vert\cdot\Vert_1$, respectively.
Furthermore,
\begin{eqnarray*}
\H&=&\{\u\in \L^2; \:\div \u=0,\:
\u\cdot {\textbf{n}}=0 \mbox{ on }\partial\Om\},\\
 \V&=&\{\u\in
\H^1;\: \div \u=0,\: \u=0  \mbox{ on }\partial\Om\}.
\end{eqnarray*}
\item  
From now on, $C>0$ will denote different constants, depending only
on the  data of the problem.
\item
We consider  the following zero-mean
spaces:
$$\begin{array}{c}
 L^2_*=\left\{\psi\in L^2; \: \jnt_\Om \psi= 0 \right\},
 \\[0.5em]
  H^1_*=H^1\cap L^2_* ,
\\[0.2em]
H^k_1=\left\{\psi\in H^k\cap L^2_* ; \:\partial_n \psi |_{\partial\Om }=0 \right\},
\qquad k= 2, 3,
\end{array}$$
and the dual space $H_*^{-1}=(H^1_*)'$.
\end{itemize}
\section{The model}\label{TheModel}
 Vesicle membranes within  incompressible viscous fluids in a bounded $3D$ domain $\Om\subset\mathbb{R}^3$ during a time interval $[0,+\infty)$, are modeled via a phase-field function $\phi:\overline \Om\times [0,+\infty) \to \mathbb{R}$  such that two stable values $\phi=1$ and $\phi=-1$ represent the exterior and interior of vesicle membranes, respectively, and $-1<\phi<1$ in the interface. Then, we will analyze the  case where the so-called {\it bending energy} $\mathcal{E}_b(\phi)$ is given by a simplified elastic Willmore energy plus a penalization of the surface area constraint \cite{DuLiuWang}:

\begin{equation}\label{energy}
\mathcal{E}_b(\phi)=\fra{1}{2\eps}\jnt_\Omega ( -\eps\Delta \phi+\fra 1\eps F'(\phi))^2\, d\x+\fra{M}{2}(\mathcal{A}(\phi)-\alpha)^2
\end{equation}
where  $F' (\phi)=(\phi^2-1)\phi$ denotes the derivative of the Ginzburg-Landau potential  
$$F (\phi)=\fra{1}{4} ( \phi^2-1)^2,
$$ $M>0$ is a penalization constant, $\eps>0$ is related to the interface width, and
$$\mathcal{A}(\phi)=\jnt_\Omega\left(\fra{\eps}{2}\vert \grad\phi\vert^2+\fra{1}{\eps}F(\phi)\right)\, d\x,$$
is an approximation of the surface area.
\begin{remark} Other possible approximation of the surface area is to consider $$\mathcal{A}(\phi)=\jnt_\Omega\fra{\eps}{2}\vert \grad\phi\vert^2\, d\x$$ (see \cite{CampHdezMach}). 
The same results of this paper may be extended to this case.
\end{remark}

To model the dynamic in time of the vesicle-fluid interactions,  we will introduce the following Navier-Stokes-Cahn-Hilliard equations in $Q=\Om\times (0,+\infty)$:
\begin{eqnarray}
\partial_t
\u+(\u\cdot\grad)\u-\nu \Delta \u -\lambda w \grad\phi+\grad q=0,\label{P1}
\\ \div \u=0, \label{P2}\\
 \partial_t \phi+\u\cdot\grad\phi -\gamma  \Delta w=0, \label{P3}
\end{eqnarray}
where $$w:=\fra{\delta \mathcal{E}_b(\phi)}{\delta \phi}$$ is 
the chemical potential, see (\ref{w-def}) below.
The coefficients $\nu>0$,
$\lambda>0$ and $\gamma>0$  depend on viscosity, elasticity and mobility, respectively. The system (\ref{P1})-(\ref{P3})  is completed  with the boundary
conditions
\begin{equation}\label{ccs0}
 \u|_{\partial\Om}=0, 
 \quad\partial_n\phi|_{\partial\Om}=0,
 \quad \partial_n(\Delta\phi)|_{\partial\Om}=0,
 \quad \partial_n w|_{\partial\Om}=0, 
\end{equation}
and  the initial conditions
    \begin{equation}\label{cis0} 
 \u|_{t=0}=\u_0, \qquad \phi|_{t=0}=
 \phi_0\quad\mbox{in }\Om .
\end{equation}
For compatibility, we will assume $\u_0|_{\partial\Om}= 0$ with $\nabla\cdot \u_0=0$ in $\Om$ and $\partial_n\phi_0|_{\partial\Om}=0$.

By integrating the phase-equation (\ref{P3}), using the free-divergence $\div \u=0$, the non-slip condition $ \u|_{\partial\Om}=0$, and the  last boundary condition of (\ref{ccs0}), $\partial_n w|_{\partial\Om}=0$,  it is easy to deduce that  the total volume of $\phi$ in $\Om$ is conserved:
\begin{equation}\label{phi-cons}
\displaystyle\frac{d}{dt}\int_\Om \phi(x,t) \;d\x= 0.
\end{equation}
%
In order to complete the problem (\ref{P1})-(\ref{P3}), we compute the variational derivative $w=\fra{\delta \mathcal{E}_b(\phi)}{\delta \phi}$. 

On the one hand,  for all $\phi,\overline{\phi}\in H^1$,  
$$
\Big\langle \fra{\delta \mathcal{A}(\phi)}{\delta \phi}, \overline{\phi}\Big\rangle=
\jnt_\Omega \eps \grad\phi\cdot\grad\overline{\phi}+\fra{1}{\eps}F'(\phi)\overline{\phi},
$$
hence if $\phi\in H^2$ and $\partial_n\phi\vert_{\partial \Omega}=0$,  integrating by parts, we can identify 
\begin{equation}\label{mu}
\mu(\phi) :=\fra{\delta \mathcal{A}(\phi)}{\delta \phi}=-\eps\Delta\phi+\fra{1}{\eps}F'(\phi).
\end{equation}
Note that, if $\phi\in H^3$ and $\partial_n\phi|_{\partial\Om}=0$, then $\partial_n \Delta \phi|_{\partial\Om}=0$ is equivalent to  $\partial_n\mu(\phi)|_{\partial\Om}=0$.

On the other hand, by using (\ref{mu}), the bending energy (\ref{energy}) can be rewritten as 
$$\label{energy2}
\mathcal{E}_b(\phi)=\fra{1}{2\eps}\jnt_\Omega\mu(\phi)^2\, d\x+\fra{1}{2}M(\mathcal{A}(\phi)-\alpha)^2 .
$$
Then,  for all $\phi,\overline{\phi}\in H^2$,  
$$\begin{array}{l}
\Big\langle \fra{\delta \mathcal{E}_b(\phi)}{\delta \phi}, \overline{\phi}\Big\rangle=
\fra{1}{\eps}\jnt_\Omega\mu(\phi)( -\eps\Delta \overline{\phi}+\fra1\eps F''(\phi)\overline{\phi})+M(\mathcal{A}(\phi)-\alpha)\Big\langle \fra{\delta \mathcal{A}(\phi)}{\delta \phi}, \overline{\phi}\Big\rangle.
\end{array}
$$
If $\phi\in H^4_2$ and $\overline{\phi}\in H^2_1$, after some integrations by parts, using  $\grad\mu(\phi)\cdot\n|_{\partial\Om}=0$ and $\grad\overline{\phi}\cdot\n|_{\partial\Om}=0$,  we can identify 
\begin{equation} \label{w-def}
w=\fra{\delta \mathcal{E}_b(\phi)}{\delta \phi}= -\Delta\mu(\phi)+\fra{1}{\eps^2} F''(\phi)\, \mu(\phi)+M(\mathcal{A}(\phi)-\alpha)\mu(\phi).
\end{equation}
\begin{remark}
The variational derivatives $w=\fra{\delta \mathcal{E}_b(\phi)}{\delta \phi}$ and $\mu(\phi) =\fra{\delta \mathcal{A}(\phi)}{\delta \phi}$ given in (\ref{mu}) and (\ref{w-def}) have been  identified as $L^2(\Om)$-functions via the $L^2(\Om)$ scalar product.
\end{remark}

From (\ref{mu}) and (\ref{w-def}), we can decompose $w$ as
\begin{equation}\label{w-computed}
\begin{array}{l}
w
=\eps \Delta^2\phi+G(\phi)
\end{array}
\end{equation}
where
\begin{equation}\label{opL}
\begin{array}{l}
G(\phi):=-\fra{1}{\eps}\Delta (F'(\phi))+ \fra{1}{\eps^2} F''(\phi)\, \mu(\phi)+ M(\mathcal{A}(\phi)-\alpha)\mu(\phi)\\
\phantom{G(\phi):}
=-\fra{2}{\eps}F''(\phi)\Delta \phi -\fra{1}{\eps}F'''(\phi)\vert\grad \phi\vert^2
+\fra{1}{\eps^3}F''(\phi)F'(\phi)
\\
\phantom{G(\phi):=}
+ M(\mathcal{A}(\phi)-\alpha)(-\eps\Delta\phi+\fra{1}{\eps}F'(\phi)).
\end{array}
\end{equation}
\subsection{A equivalent ``zero-mean" problem}
With the aim to use the conservative property given  in (\ref{phi-cons}), if we define
 $$m_0=\langle\phi_0\rangle:=\fra{1}{\vert\Om\vert} \int_\Om \phi_0(x)\,d\x,$$ 
 we can introduce the following zero-mean variables:
$$
\psi(x,t):=\phi(x,t)-m_0\quad\mbox{ and }\quad z:=w-\langle G(\phi)\rangle.
$$
Observe that, 
$$
\jnt_\Om \p(t,\cdot)\;d\x=0\quad \hbox{and} \quad \jnt_\Om z(t,\cdot)\;d\x=0.
$$ 
In fact, since  $\partial_n\phi|_{\partial\Om}=0$, in particular $\partial_n(F'(\phi))|_{\partial\Om}=0$, hence $\jnt_\Om -\Delta (F'(\phi)) \;d\x=0$.  
Therefore, integrating (\ref{w-computed}) and (\ref{opL}),
\begin{equation}\label{G-mean}
\langle w\rangle=\langle G(\phi)\rangle= \fra{1}{\eps^2}\langle F''(\phi)\, \mu(\phi)\rangle+ M(\mathcal{A}(\phi)-\alpha)\langle\mu(\phi)\rangle.
\end{equation}
Reciprocally, given $(\psi,z)$ we can recover $(\phi,w)$ as $\phi=\psi + m_0$ and $w=z+\langle G(\phi)\rangle$.

By rewriting the equations (\ref{P1})-(\ref{P3}) and (\ref{w-computed}) in these new variables $(\psi,z)$ we arrive at  
\begin{eqnarray}
\partial_t
\u+(\u\cdot\grad)\u-\nu \Delta \u -\lambda \, z\, \grad\psi+\grad \widetilde{q}=0,\label{P1s}
\\ \div \u=0, \label{P2s}\\
 \partial_t \psi+\u\cdot\grad\psi -\gamma  \Delta z=0, \label{P3s}
 \\
\eps  \Delta^2\psi+\overline{G}(\p)-z=0,\label{P4s}
\end{eqnarray}
where 
$$
\overline{G}(\p):=G(\psi+m_0)-\langle G(\psi+m_0)\rangle,
\quad
\widetilde{q}:=q-\lambda \langle G(\psi+m_0)\rangle\p.$$
By using (\ref{opL}) and (\ref{G-mean}),
\begin{equation}\label{Gbarra}
\begin{array}{l}
\overline{G}(\p)  
=-\fra{2}{\eps}F''(\psi+m_0)\Delta \p -\fra{1}{\eps}F'''(\psi+m_0)\vert\grad \p\vert^2
+\fra{1}{\eps^3}F''(\psi+m_0)F'(\psi+m_0)
\\
\phantom{G(\phi)=}
+ M(\mathcal{A}(\psi+m_0)-\alpha)\left(-\eps\Delta\p+\fra{1}{\eps}F'(\psi+m_0)\right)
\\
\phantom{G(\phi)=}
-\fra{1}{\eps^2}\fra{1}{\vert\Om\vert} \int_\Om \left(- \eps\Delta \p + 
\fra{1}{\eps} F'(\psi+m_0) \right) F''(\psi+m_0) 
\\
\phantom{G(\phi)=}
- M(\mathcal{A}(\psi+m_0)-\alpha)\fra{1}{\vert\Om\vert} \int_\Om \left(-\eps\Delta\p+\fra{1}{\eps}F'(\psi+m_0) \right).
\end{array}
\end{equation}

System (\ref{P1s})-(\ref{P4s}) is completed  with the boundary and initial 
conditions
\begin{equation}\label{ccs}
 \u|_{\partial\Om}=0, 
 \quad\partial_n\p|_{\partial\Om}=0,
 \quad \partial_n(\Delta\p) |_{\partial\Om}=0,
 \quad \partial_n z|_{\partial\Om}=0, 
\end{equation}
    \begin{equation}\label{cis}
 \u|_{t=0}=\u_0, \qquad \p|_{t=0}=\p_0:=
 \phi_0-\langle\phi_0\rangle\quad\mbox{in }\Om .
\end{equation}
Consequently, problem (\ref{P1})-(\ref{cis0}) is equivalent to the ``zero-mean" problem $(\ref{P1s})$-$(\ref{cis})$.

Note that, by denoting  the bending energy with respect to the zero-mean unknown $\p$ as
\begin{equation}\label{identenergy1} 
\overline\mathcal{E}_b(\p)=\mathcal{E}_b(\p+m_0),
\end{equation} 
then it is easy to check that  
\begin{equation}\label{identenergy2}
z=\fra{\delta \overline\mathcal{E}_b(\psi)}{\delta \psi}=\eps \Delta^2\psi+
\overline{G}(\p)\quad \forall \psi\in H^4_2,
\end{equation} 
where the identification  (\ref{identenergy2}) has been computed via the $L_*^2(\Om)$-scalar product. Latter,  we will give a weak sense to this identification in order to define the concept of weak solution.
\section{Some preliminary results} \label{sec:prelim}
We assume $\Om$ sufficiently regular (for instance $\partial\Omega\in C^3$) in order to have the $H^3$-regularity of the following second order elliptic Poisson-Neumann problem: 
$$
\left\{
\begin{array}{l}
-\Delta v=f\quad\mbox{in }\Omega\\
\partial_nv\vert_{\partial\Omega}=0,\quad 
 \jnt_\Om v\;d\x=0,
\end{array}
\right.
\leqno{\rm (P_0)}
$$
where $f\in L_*^2(\Om)$. In fact, the $H^2$ and $H^3$-regularity of $
{\rm (P_0)}$ provide, respectively,  the existence of some constants  $C>0$, such that: 
\begin{equation}\label{n4w}
\Vert v\Vert_2\leq C  \vert\Delta v\vert_2\quad\forall \, v\in H^2_1,
\quad
\Vert v\Vert_3\leq C  \Vert\Delta v\Vert_1\quad\forall\, v\in H^3_1.
\end{equation}
In particular, since $\Delta v \in L^2_*$, by applying the Poincare inequality for zero-mean functions, we have that
\begin{equation}\label{n4wA}
\Vert v\Vert_3\leq C  \vert\nabla\Delta v\vert_2\quad\forall\, v\in H^3_1.
\end{equation}
 
 \
 
In order to define the inner product in $H^{-1}_*$, observe that if $r\in H^{-1}_*$, then applying the Lax-Milgram's Lemma, there exists a unique $u_r\in H^{1}_*$ such that 
$\langle r,v\rangle_{H^{-1}_*,H^{1}_*}=(u_r,v)_{H^{1}_*}=(\grad u_r,\grad v)_{L^2}$ for all $v\in H^{1}_*$. Therefore, $\Vert r\Vert_{H^{-1}_*}=\Vert \grad u_r\Vert_{L^2}$ and $(r,s)_{H^{-1}_*}=(\grad u_r,\grad u_s)_{L^2}$ for all $s\in H^{-1}_*$. 

In order to give a weak sense to $\Delta^2 v$ together with the boundary condition $\partial_n(\Delta v)|_{\partial \Omega}=0$ for any $v\in H^3_1$, we introduce the  operator $A:D(A)=H^3_1\subset H^{-1}_* \mapsto H^{-1}_*$ as follows:
\begin{equation}\label{Aoperator}
\langle A\psi, \widetilde{\psi}\rangle_{H^{-1}_*,H^1_*}=-(\grad\Delta\psi, \grad \widetilde{\psi})_{L^2}\quad \forall\, \psi\in H_1^3, \ \forall\,\widetilde{\psi}\in H^1_*.
\end{equation}
In particular, denoting $r=A\psi$ and $s=A \widetilde\psi$, then $u_r=-\Delta \psi$ and $u_r=-\Delta \widetilde\psi$, hence the inner product in $H^{-1}_*$ remains $(A\psi, A\widetilde\psi)_{H^{-1}_*}=
(\grad\Delta\psi, \grad \Delta\widetilde{\psi})_{L^2}$.

Observe that $A\in \mathcal{L}(H^3_1, H^{-1}_*)$ and it is self-adjoint and positive definite. Indeed, $A$ is continuous because
$$
\Vert A\psi\Vert_{H^{-1}_*}=\sup_{ \widetilde{\psi}}\frac{\langle A\psi, \widetilde{\psi}\rangle_{H^{-1}_*,H^1_*}}{\Vert  \widetilde{\psi}\Vert_{H^{1}_*}}
= 
\sup_{ \widetilde{\psi}}\frac{-(\grad\Delta\psi, \grad \widetilde{\psi})_{L^2}}{\vert  \grad\widetilde{\psi}\vert_{2}}
$$
In particular, taking $\widetilde{\psi}=\Delta\psi \in H^{1}_*$, one has 
\begin{equation} \label{norm-A}
\Vert A\psi\Vert_{H^{-1}_*}=\Vert  \grad \Delta{\psi}\Vert_{L^2}
\leq C\Vert \psi\Vert_3.
\end{equation}
On the other hand, for all $ \psi,\, \widetilde{\psi}\in H_1^3$, since $\partial_n \psi\vert_{\partial\Omega}=\partial_n \widetilde{\psi}\vert_{\partial\Omega}=0$, we have that
$$\langle A\psi, \widetilde{\psi}\rangle_{H^{-1}_*,H^1_*}=
-(\grad\Delta\psi, \grad \widetilde{\psi})_{L^2}=
(\Delta \psi, \Delta\widetilde{\psi})_{L^2}=
-(\grad\psi, \grad\Delta\widetilde{\psi})_{L^2}=
\langle A\widetilde{\psi}, \psi\rangle_{H^{-1}_*,H^1_*}.$$
 Therefore, $A$ is symmetric. In particular, $\langle A\psi, \psi\rangle_{H^{-1}_*,H^1_*}=\Vert\Delta \psi\Vert_{L^2}^2\geq C \Vert \psi\Vert_{2}^2$ for all $ \psi\in H_1^3$, hence $A$ is positive definite. Finally, we are going to prove that there exists $A^{-1}\in \mathcal{L}( H^{-1}_*,H^3_1)$ by applying the Banach-Necas-Babuska's Theorem (see for instance \cite{ErnGuermond}). For this, we  consider the following continuous bilinear form $a(\cdot,\cdot): H_1^3\times H_*^1 \mapsto \mathbb{R}$:
$$a(\psi,\,\widetilde{\psi})=\langle A\psi, \widetilde{\psi}\rangle_{H^{-1}_*,H^1_*}
=-(\grad\Delta\psi, \grad \widetilde{\psi})_{L^2}\quad\forall
\psi\in H_1^3,\:\forall\widetilde{\psi}\in H_*^1.$$
Owing to (\ref{n4wA}) and (\ref{norm-A}), there exists a constant $\beta>0$ such that
 $$\displaystyle\sup_{\widetilde{\psi}\in H_*^1}\frac{a(\psi,\,\widetilde{\psi})}{\vert\grad \widetilde{\psi}\vert_2}= \Vert A\psi\Vert_{H^{-1}_*} \geq \beta \Vert\p\Vert_3.$$
 On the other hand, if we assume $a(\psi,\,\widetilde{\psi})=0$ for all $\psi\in H_1^3$, then taking in particular $\psi\in H_1^3$ the solution of problem ${\rm (P_0)}$ for $f=\widetilde{\psi}$, then $-\Delta\psi=\widetilde{\psi}$ and
 $$
 0=a(\psi,\,\widetilde{\psi})=
 -(\grad\Delta\psi, \grad \widetilde{\psi})_{L^2}
 =|\grad \widetilde{\psi} |^2_2
 $$
 hence $\widetilde{\psi}\equiv Cte.$ But, since $\int_\Om \widetilde{\psi} =0$, then $\widetilde{\psi} \equiv0$. Therefore, the Banach-Necas-Babuska's Theorem implies that there exists $A^{-1}\in \mathcal{L}( H^{-1}_*,H^3_1)$. Moreover, $A$ is self-adjoint.

\

The following lemma shows two compactness results of Aubin-Lions type, see \cite{Simon}.
\begin{lemma}\label{simon}
Let us consider $T>0$ and three Banach spaces such that $X\subset B\subset Y$ with continuous imbedding $B\rightarrow Y$ and continuous and compact imbedding $X\mapsto B$.
\begin{itemize}
\item If the set $F$ is bounded in $L^p(0,T;X)$ where $p<\infty$ and $\partial_t F=\{\partial_tf: f\in F\}$ is bounded in $L^1(0,T;Y)$, then $F$ is relatively compact in $L^p(0,T;B)$.
\item If the set $F$ is bounded in $L^\infty(0,T;X)$ and $\partial_t F$ is bounded in $L^q(0,T;X)$ with $q>1$, then $F$ is relatively compact in $C([0,T];B)$.
\end{itemize}
\end{lemma}

\

Along this paper, we will use repeatly the following classical interpolation and Sobolev
inequalities (for $3D$ domains):
$$\vert v\vert_6\leq C\Vert v\Vert_1, \quad
\vert v\vert_3\leq C \vert v\vert_2^{1/2}\Vert
v\Vert_1^{1/2}\quad\forall v \in H^1$$ 
and the Gagliardo-Nirenberg inequality
$$\vert v\vert_\infty\leq C \Vert v\Vert_1^{1/2}\Vert v\Vert_2^{1/2}\quad\forall v \in H^2.$$
For the last inequality, see for example (\cite{PonceTiti}, p.\ 334).

\

The following Lemma gives some global Lipschitz properties of $F(\phi)$ its derivatives and $\mathcal{A}(\phi)$ into $H^2$-bounded sets.
\begin{lemma}\label{DiferenciaF}
Let us consider $K>0$ a constant and any functions $\phi_i\in H^2(\Om)$, with $\Vert\phi_i\Vert_2\leq K$, for $i=1,2$, then the following inequalities are fulfilled for any $p$ with $1\le p\le \infty$:
$$\begin{array}{c}
\vert F(\phi_i)\vert_p\leq C(K),\quad
\vert F'(\phi_i)\vert_p\leq C(K),\quad
\vert F''(\phi_i)\vert_p\leq C(K),\quad
\vert F'''(\phi_i)\vert_p\leq C(K)\quad i=1,2,\\
\vert F(\phi_1)-F(\phi_2)\vert_p\leq C(K)\vert \phi_1-\phi_2\vert_p,\quad
\vert F'(\phi_1)-F'(\phi_2)\vert_p\leq C(K)\vert \phi_1-\phi_2\vert_p,\\
\vert F''(\phi_1)-F''(\phi_2)\vert_p\leq C(K)\vert \phi_1-\phi_2\vert_p,\quad
\vert F'''(\phi_1)-F'''(\phi_2)\vert_p\leq C(K)\vert \phi_1-\phi_2\vert_p,\\
\vert \mathcal{A}(\phi_1)-\mathcal{A}(\phi_2)\vert\leq C(K)\Vert \phi_1-\phi_2\Vert_1,
\end{array}
$$
where $C(K)>0$ are different constants depending on $K$. 
\end{lemma}
\begin{proof}
Let us remember that $F(\phi)=\frac{1}{4}(\phi^2-1)^2$, $F'(\phi)=(\phi^2-1)\phi$, $F''(\phi)=3\phi^2-1$, $F'''(\phi)=6\phi$ and $\mathcal{A}(\phi)=\jnt_\Omega\left(\fra{\eps}{2}\vert \grad\phi\vert^2+\fra{1}{\eps}F(\phi)\right)\, d\x$.
We prove the first inequality:
$$\vert F(\phi_i)\vert_p= \frac{1}{4}\vert \phi^4+2\phi^2-1\vert_p
\leq C(\vert \phi\vert_\infty^4+2\vert\phi\vert_\infty^2+1)
\leq C(\Vert \phi\Vert_2^4+2\Vert\phi\Vert_2^2+1)\leq C(K),$$
The second, third and fourth inequalities can be obtained in an analogous way. We prove now the fifth one:
$$\begin{array}{l}
\vert F(\phi_1)-F(\phi_2)\vert_p\le \vert F'(\theta\phi_1+(1-\theta)\phi_2)\vert_\infty 
\vert \phi_1-\phi_2\vert_p
\leq C(K)\vert\phi_1-\phi_2\vert_p.$$
\end{array}
$$
Sixth, seventh and eighth  inequalities can be proved in an analogous way. Finally, let us see the ninth one:
$$\begin{array}{c}
\displaystyle
\vert \mathcal{A}(\phi_1)-\mathcal{A}(\phi_2)\vert\leq 
\frac12\vert |\nabla\phi_1|_2^2-|\nabla\phi_2|_2^2\vert
+ \int_\Om |F(\phi_1) - F(\phi_2)|
\\
\displaystyle
\leq 
\frac12\vert\grad(\phi_1+\phi_2)\vert_2\vert\grad(\phi_1-\phi_2)\vert_2
+C(K)\vert\phi_1-\phi_2\vert_1
\\
\displaystyle
\leq C(K)\Vert\phi_1-\phi_2\Vert_1.
\end{array}$$
\end{proof}
The following Lemma gives the global Lipschitz property of $G(\phi)$ into $H^2$-bounded sets.
\begin{lemma}\label{DiferenciaG}
The map $G=G(\phi)$ given in (\ref{opL}) is well-posed from $H^2(\Om)$ into $L^2(\Om)$. Moreover, for any functions $\phi_i\in H^2(\Om)$, with $\Vert\phi_i\Vert_2\leq K$, for $i=1,2$, there exists $C=C(K)>0$ such that
$$\vert G(\phi_1)-G(\phi_2)\vert_2\leq C(K)\, \Vert\phi_1-\phi_2\Vert_2.$$
\end{lemma}
\begin{remark}\label{remark6}
In particular,  since $G(0)=0$, Lemma~\ref{DiferenciaG} implies $\vert G(\phi)\vert_2\leq C(K) \Vert\phi\Vert_2\leq C(K)$.
\end{remark}

\begin{proof} 
From (\ref{opL}) 
\begin{equation}\label{prev-estim}
\begin{array}{l}
\vert G(\phi_1)-G(\phi_2)\vert _2\leq\fra{2}{\eps}\vert F''(\phi_1)\Delta\phi_1-F''(\phi_2)\Delta \phi_2\vert _2
\\\qquad
+\fra{1}{\eps}\vert  F'''(\phi_1)\vert\grad\phi_1\vert^2-F'''(\phi_2)\vert\grad\phi_2\vert^2\vert_2
+\fra{1}{\eps^3}\vert F'(\phi_1)F''(\phi_1)-F'(\phi_2)F''(\phi_2)\vert_2
\\\qquad
+\eps M\vert  (\mathcal{A}(\phi_1)-\alpha)\Delta\phi_1-(\mathcal{A}(\phi_2)-\alpha)\Delta\phi_2\vert_2 
\\\qquad
+\fra{M}{\eps}
\vert (\mathcal{A}(\phi_1)-\alpha)F'(\phi_1)-(\mathcal{A}(\phi_2)-\alpha)F'(\phi_2)\vert_2
:=\sum_{i=1}^5 I_i.
\end{array}
\end{equation}
By taking into account Lemma \ref{DiferenciaF}, we can estimate each term $I_i$ as follows:
$$\begin{array}{l}
I_1=\fra{2}{\eps}\vert F''(\phi_1)\Delta\phi_1-F''(\phi_2)\Delta \phi_2\vert _2
\\\phantom{I_1}
\leq C(
\vert(F''(\phi_1)-F''(\phi_2))\Delta\phi_1\vert_2+\vert F''(\phi_2)(\Delta\phi_1-\Delta \phi_2)\vert_2)
\\\phantom{I_1}
\leq C(
\vert\Delta\phi_1\vert_2\vert F''(\phi_1)-F''(\phi_2)\vert_\infty+
\vert F''(\phi_2)\vert_\infty\vert\Delta\phi_1-\Delta \phi_2\vert_2)
\\\phantom{I_1}
\leq 
C(\Vert\phi_1\Vert_2 C(K)\vert \phi_1-\phi_2\vert_\infty+
C(K)\Vert\phi_1- \phi_2\Vert_2)
\\\phantom{I_1}
\leq C(K)\Vert\phi_1-\phi_2\Vert_2.
\end{array}$$
$$\begin{array}{l}
I_2=\fra{1}{\eps}\vert  F'''(\phi_1)\vert\grad\phi_1\vert^2-F'''(\phi_2)\vert\grad\phi_2\vert^2\vert_2
\\\phantom{I_2}
\leq C (\vert (F'''(\phi_1)- F'''(\phi_2))\vert\grad\phi_1\vert^2\vert_2+\vert F'''(\phi_2)(\vert\grad\phi_1\vert^2-\vert\grad\phi_2\vert^2)\vert_2)
\\\phantom{I_2}
\leq C (\vert\grad\phi_1\vert^2_6\vert F'''(\phi_1)- F'''(\phi_2)\vert_6+\vert F'''(\phi_2)\vert_6\vert\grad\phi_1+\grad\phi_2\vert_6\vert\grad\phi_1-\grad\phi_2\vert_6)
\\\phantom{I_2}
\leq C (\Vert\phi_1\Vert^2_2C(K)\vert \phi_1- \phi_2\vert_6+C(K)\Vert\phi_1+\phi_2\Vert_2\Vert\phi_1-\phi_2\Vert_2)
\\\phantom{I_1}
\leq C(K)\Vert\phi_1-\phi_2\Vert_2.
\end{array}$$
$$\begin{array}{l}
I_3=
\fra{1}{\eps^3}\vert F'(\phi_1)F''(\phi_1)-F'(\phi_2)F''(\phi_2)\vert_2
\\\phantom{I_3}
\leq C(\vert (F'(\phi_1)-F'(\phi_2))F''(\phi_1)\vert_2
+ \vert F'(\phi_2)(F''(\phi_1)-F''(\phi_2))\vert_2)
\\\phantom{I_3}
\leq C(\vert F'(\phi_1)-F'(\phi_2)\vert_6\vert F''(\phi_1)\vert_3
+ \vert F'(\phi_2)\vert_3\vert F''(\phi_1)-F''(\phi_2)\vert_6)
\\\phantom{I_2}
\leq C(K)\Vert\phi_1-\phi_2\Vert_1.
\end{array}$$
To estimate $I_4$ and $I_5$ observe that
$$
\vert \mathcal{A}(\phi_2)-\alpha\vert\leq C( \vert\grad\phi_2\vert_2^2+\vert\phi_2^2-1\vert_2^2)
\leq C(\Vert\phi_2\Vert_1^2+\Vert\phi_2\Vert_2^4+1)\leq C(K),
$$
therefore,
$$\begin{array}{l}
I_4=
\eps M\vert  (\mathcal{A}(\phi_1)-\alpha)\Delta\phi_1-(\mathcal{A}(\phi_2)-\alpha)\Delta\phi_2\vert_2 
\\\phantom{I_4}
\leq C(\vert   \mathcal{A}(\phi_1)-\mathcal{A}(\phi_2)\vert \, \vert\Delta\phi_1\vert_2+\vert \mathcal{A}(\phi_2)-\alpha\vert \, \vert\Delta(\phi_1-\phi_2 )\vert_2)
\\\phantom{I_4}
\leq C(K)(\Vert\phi_1-\phi_2\Vert_1+\Vert\phi_1-\phi_2\Vert_2)
\leq C(K)\Vert\phi_1-\phi_2\Vert_2,
\end{array}$$
and finally,
$$\begin{array}{l}
I_5=
\fra{M}{\eps}
\vert (\mathcal{A}(\phi_1)-\alpha)F'(\phi_1)-(\mathcal{A}(\phi_2)-\alpha)F'(\phi_2)\vert_2
\\\phantom{I_5}
\leq C(
\vert \mathcal{A}(\phi_1)-\mathcal{A}(\phi_2)\vert\vert F'(\phi_1)\vert_2+\vert \mathcal{A}(\phi_2)-\alpha\vert\vert F'(\phi_1)-F'(\phi_2)\vert_2)
\\\phantom{I_5}
\leq C(K)\Vert\phi_1-\phi_2\Vert_1.
\end{array}$$
Plugging all previous estimates in (\ref{prev-estim}),  the proof of Lemma is finished.
\end{proof}
 \section{A new Lojasiewicz-Simon's result} \label{LSR}
 The use of Lojasiewicz-Simon inequalities is a classical procedure to study the convergence of trajectories at infinite time in dissipative systems. It is not easy to find in the literature a rigorous demonstrations of these types of inequalities associated to various Euler-Lagrange equations. Here,  a particular Lojasiewicz-Simon's inequality associated to the critical points of the bending energy $\overline\mathcal{E}_b(\p)$ for zero-mean functions $\p$ is deduced, by using the Theorem~\ref{huang} presented below, which is a simplified version of Theorem 4.2, p.\ 41 of \cite{Huang}. 
 We make an extension of the Lemma 4.4 of \cite{SegattiWu} (which is applied to a second order elliptic problem) to the fourth-order elliptic problem that determines the critical points $\p_*$ of $\overline\mathcal{E}_b(\p)$  and then, it will be possible to prove a modified  Lojasiewicz-Simon's result by relaxing the hypothesis of small $\Vert\psi-\psi_*\Vert_3$ by small $\Vert\psi-\psi_*\Vert_1$ and $|\overline\mathcal{E}_b(\p)-\overline\mathcal{E}_b({\p}_*)|$. 

We begin by recalling two definitions:
\begin{definition}[\cite{Huang} p.\ 22] Let $U$ be an open subset in a real Banach space $(X, \Vert\cdot\Vert_X)$ and Y a subspace of the dual space $X'$.
A continuos map $\mathcal{M}:U\mapsto Y$ is called a {\bf gradient map} if there exists a $C^1$ functional $\mathcal {E}:U\mapsto \mathbb{R}$
such that $\mathcal{M}(u)=\mathcal {E'}(u)$ for all $u\in U$, i.e.
$$\mathcal {E'}(u)h=\langle\mathcal{M}(u),h\rangle\qquad\forall
u\in U, \;h\in X$$
where $\langle\cdot,\cdot\rangle$ is the canonical bilinear form on $X'\times X$.
\end{definition}
\begin{definition}[\cite{Huang} p.\ 34]
A bounded linear operator $L:X_1\mapsto X_2$ between two Banach spaces $X_1$ and $X_2$  is called 
a  {\bf Fredholm operator of index zero} if $L$ has a closed range $R(L)$, a finite dimensional kernel $N(L)$ and $\rm{dim}(N(L))=\rm{dim}(X_2/R(L))<\infty$. A $C^1$-map $\mathcal {M}:U\subset X_1\mapsto X_2$
is called 
a Fredholm map of index zero if its Fr\`echet differential at each point is a Fredholm operator of index zero.
\end{definition}
For instance, an invertible operator plus a compact operator is a  Fredholm operator of index zero (see, for example \cite{Brezis}{ pp.\ 98, 99})

We now give a simplified version of Theorem 4.2 of \cite{Huang}.
\begin{theorem}\label{huang}
Assume the following hypotheses:
\begin{itemize}
\item
Let $H$ be a Hilbert space and let $A:D(A)\subset H\mapsto H$ be a linear self-adjoint and positive definite operator. We denote  $(D(A),\langle\cdot,\cdot\rangle_A) $  the Hilbert space endowed with the scalar product $\langle u,v\rangle_A\equiv (Au,Av)_{H}$ for all $u,v\in D(A)$. Assume that the embedding $D(A) \subset H$ is continuous.
\item
Let  $ \mathcal{E}: D(A)\mapsto \mathbb{R}$ be a Fr\'echet-differentiable  functional, and let $\mathcal {M}: D(A) \mapsto H$ be an analytic gradient map associated to $\mathcal{E}$ (i.e.  $\mathcal {M}= \mathcal{E}'$ ) with the following
properties: 
\begin{itemize}
\item $\mathcal {M}$ is a Fredholm map of index zero; i.e., for each $u \in D(A)$ the bounded  linear operator $\mathcal {M}'(u) \in {\cal L}(D(A),H)$ is a Fredholm operator of index zero.
\item The map
$ {\cal R}:u\in D(A) \mapsto \mathcal {M'}(u)A^{-1} \in {\cal L}(H)$
is continuous.
\end{itemize}
\item
Let $u_*\in D(A)$ be a critical point of $\mathcal{E}(u)$, i.e. $\mathcal{E}'(u_*)=0$.
\end{itemize}
Then, there exists  positive constants $C$, $\beta$ and $\theta\in(0,1/2]$ such that for all $u\in D(A)$ with $\Vert u-u_*\Vert_{D(A)}\le \beta$, it holds
$$\vert\mathcal{E}(u)-\mathcal{E}(u_*)\vert^{1-\theta}\leq C \,\Vert\mathcal{E}'(u)\Vert_H.$$
\end{theorem}

\begin{lemma}[Lojasiewicz-Simon inequality] \label{le:L-S2} Let ${\cal S}$ be the following set of equilibrium points related to the bending energy $ \overline\mathcal{E}_b(\p)$ given in (\ref{identenergy1}):  
\begin{equation}\label{set-S}
{\cal S}=\{\p\in H^3_1(\Omega):\:-\eps (\grad\Delta\p,\grad\widetilde\p)+(\overline{G}(\p),\widetilde\p)=0\quad \forall\, \widetilde\p\in H^{1}_*\}.
\end{equation} 
 Let  ${\p}_*\in {\cal S}$ and $K>0$ fixed. Then, there exists positive constants $\beta_1$, $\beta_2$ and  $C$ and $\theta\in(0,1/2]$, 
 such that for all $\p\in H^3_1$ with  $\Vert\p\Vert_2\leq K$,  $\Vert\p-{\p}_*\Vert_1\leq \beta_1$ and $\vert \overline\mathcal{E}_b(\p)-\overline\mathcal{E}_b({\p}_*)\vert\leq\beta_2$, it holds 
 \begin{equation}\label{conclu}
 \vert \overline\mathcal{E}_b(\p)-\overline\mathcal{E}_b({\p}_*)\vert^{1-\theta}\leq C\,\Vert z(\p)\Vert_{H^{-1}_*}
\end{equation}
 where $z(\p):=\eps A\p+\overline{G}(\p)$ in the sense 
 $$
 \langle z(\p), \widetilde\p\rangle_{H^{-1}_*,H^{1}_*}=-\eps (\grad\Delta\p,\grad\widetilde\p)+(\overline{G}(\p),\widetilde\p), \quad \forall\, \widetilde\p\in H^{1}_*.$$
\end{lemma}
\begin{proof} 
	
\noindent{\bf Step 1: } {\it There exists $\beta>0$, $ C>0$ and $\theta\in(0,1/2]$ (depending on ${\p}_*$) such that for all $\p\in H^3_1$ with  $\Vert\p-{\p}_*\Vert_3\leq \beta$, then (\ref{conclu}) holds.}
\medskip

The proof of this step  is based on Theorem~\ref{huang}, choising the spaces $H\equiv H^{-1}_*(\Om)$, $D(A)\equiv H_1^3(\Om)$ and by taking the following operators: 
\begin{itemize}
\item
$A:\p\in H_1^3\mapsto A\p \in H^{-1}_*$ defined in (\ref{Aoperator}). In particular, we have  
$$\langle \p,\xi \rangle_A :=(A\p,A\xi)_{H^{-1}_*}=(\grad\Delta \p,\grad\Delta \xi)_{L^2}\quad  \forall \, \xi,\psi\in H_1^3.$$
Note that $\langle\cdot,\cdot\rangle_A$ is an inner product in $H_1^3$ owing to (\ref{n4wA}).
\item
$ \mathcal{E}:\p\in H_1^3\mapsto  \mathcal{E}(\p)=\overline\mathcal{E}_b(\p)\in \mathbb{R}$,
 \item
$\mathcal{E}':\p\in H_1^3\mapsto \mathcal{E}'(\p)\in (H_1^3)'$ defined as:  
 $$\begin{array}{l}
 \langle \mathcal{E}'(\p),\widetilde{\p} \rangle_{(H_1^3)',H_1^3}=\eps(\Delta\p,\Delta\widetilde{\p})_{L^2}+(\overline{G}(\p),\widetilde{\p})_{L^2} \\
 \phantom{\langle \mathcal{E}'(\p),\widetilde{\p} \rangle_{(H_1^3)',H_1^3}}=-\eps(\grad\Delta\p,\grad\widetilde{\p})_{L^2}+(\overline{G}(\p),\widetilde{\p})_{L^2}\quad\forall\, \p,\,\widetilde{\p}\in H_1^3,
 \end{array}$$
 \item
 
 $ \mathcal{M}\equiv\mathcal{E}':\p\in H_1^3\mapsto  \mathcal{M}(\p)\in H^{-1}_*$
is an extension of $\mathcal{E}'(\p)$ by density,  defined as:
$$
\langle \mathcal{M}(\p),\widetilde{\p} \rangle_{H_*^{-1},H_*^1}=
 \langle\eps A\p,\widetilde{\p}\rangle_{H_*^{-1},H_*^1}+(\overline{G}(\p),\widetilde{\p})_{L^2}\quad\forall\, \p\in H_1^3,\, \forall\widetilde{\p}\in H_*^1,\,
 $$
 \item 
 $\mathcal{M}'(\p):\xi\in H_1^3\mapsto\mathcal{M}'(\p)(\xi):=\eps A\xi+\overline{G}'(\p)(\xi)\in H^{-1}_*$, 
 with 
\begin{equation}\label{G-prima-eq}
\begin{array}{l}
G'(\p)(\xi)=
\\
-\fra{2}{\eps}
\Big[F''(\p+m_0)\Delta\xi + 
F'''(\p+m_0) ( \Delta\p\, \xi + \grad\p\cdot\grad\xi)
+\frac{1}{2}F^{iv)}(\p+m_0) \vert\grad\p\vert^2\xi
\Big]
\\
+\fra{1}{\eps^3}[F''(\p+m_0)^2+F'''(\p+m_0)F'(\p+m_0)]\xi
\\
+M\left(-\eps\Delta \p+\fra{1}{\eps}F'(\p+m_0)\right)\jnt_\Om\left(\eps\grad\p\cdot\grad\xi+\frac{1}{\eps}F'(\p+m_0)\xi\right)\,d\x 
\\
+M(\mathcal{A}(\p+m_0)-\alpha)\left(-\eps\Delta\xi+\fra{1}{\eps}F''(\p+m_0)\xi \right).
\end{array}
\end{equation}
\end{itemize}
Note that $\mathcal{M}'(\p)$ is a Fredholm operator of index zero, because $\mathcal{M}'(\p)$ is the sum of the invertible operator  $\eps A$ and the compact operator $\overline{G}'(\p):H_1^3\mapsto H^{-1}_*$. Indeed, by using Lemma~\ref{DiferenciaF}, if $\Vert\xi\Vert_{H_1^3}\leq C_1$, then $ \vert \overline{G}'(\p)(\xi)\vert_2\leq C_2$.
  
 On the other hand,  the map ${\cal R}:\p\in H_1^3\mapsto\mathcal{M}'(\p)A^{-1}\in\mathcal{L}(H^{-1}_*)$ is well-posed because  $A^{-1}\in \mathcal{L}(H^{-1}_*;H_1^3)$ and 
$ \mathcal{M}'(\p)\in \mathcal{L}(H_1^3;H^{-1}_*)$. It remains to prove that ${\cal R}$ is (sequentially)  continuous. Indeed, let  $\p_n\rightarrow \p$ in $H_1^3$ as $n\to \infty$. Then,
$$\begin{array}{l}
\Vert {\cal R}(\p_n)-{\cal R}(\p) \Vert_{\mathcal{L}(H^{-1}_*)}=
\Vert\mathcal{M}'(\p_n)A^{-1}-\mathcal{M}'(\p)A^{-1}\Vert_{\mathcal{L}(H^{-1}_*)}
\\\phantom{\Vert {\cal R}(\p_n)-{\cal R}(\p) \Vert_{\mathcal{L}(H^{-1}_*)}}
\leq
\Vert\mathcal{M}'(\p_n)-\mathcal{M}'(\p)\Vert_{\mathcal{L}(H_1^3;H^{-1}_*)}
\Vert A^{-1}\Vert_{\mathcal{L}(H^{-1}_*;H_1^3)}
\end{array}
$$
In order to bound $\Vert\mathcal{M}'(\p_n)-\mathcal{M}'(\p)\Vert_{\mathcal{L}(H_1^3;H^{-1}_*)}
=\Vert \overline G'(\p_n)- \overline G'(\p)\Vert_{\mathcal{L}(H_1^3;H^{-1}_*)}$, by using (\ref{G-prima-eq}) we observe that$$\begin{array}{l}
G'(\p_n)(\xi)-G'(\p)(\xi)\\
=-\fra{2}{\eps}(F''(\p_n+m_0)-F''(\p+m_0))\Delta\xi
-\fra{1}{\eps}F^{iv)}(\p_n+m_0)(\grad\p_n+\grad\p)\cdot(\grad\p_n-\grad\p)
\xi
\\
+\fra{1}{\eps^3}(F''(\p_n+m_0)+F''(\p+m_0))(F''(\p_n+m_0)-F''(\p+m_0))\xi
\\
-\fra{2}{\eps}\Big[F'''(\p_n+m_0)\Delta(\p_n-\p)+(F'''(\p_n+m_0)-F'''(\p+m_0))\Delta\p\Big]\xi
\\
-\fra{2}{\eps}\Big[F'''(\p_n+m_0)\grad(\p_n-\p)+(F'''(\p_n+m_0)-F'''(\p+m_0))\grad\p\Big]\cdot\grad\xi
\\
+\fra{1}{\eps^3}(F'''(\p_n+m_0)-F'''(\p+m_0))F'(\p_n+m_0)\xi
\\
+\fra{1}{\eps^3}F'''(\p+m_0)(F'(\p_n+m_0)-F'(\p+m_0))\xi\\
+{M}\left(-\eps\Delta \p_n+\fra{1}{\eps}F'(\p+m_0)\right)\jnt_\Om\left(\grad(\p_n-\p)\cdot\grad\xi+\fra{1}{\eps}(F'(\p_n+m_0)-F'(\p+m_0))\right)\,d\x 
\\
+\left(-\eps\Delta (\p_n-\p)+\fra{1}{\eps}(F'(\p_n+m_0)-F'(\p+m_0))\right)\jnt_\Om\left(\eps\grad\p\cdot\grad\xi+\fra{1}{\eps}F'(\p+m_0)\xi \right)\,d\x 
\\
+M(\mathcal{A}(\p_n+m_0)-\mathcal{A}(\p+m_0))\left(-\eps\Delta\xi+\fra{1}{\eps}F''(\p+m_0)\xi \right)
\\
+M(\mathcal{A}(\p+m_0)-\alpha)\fra{1}{\eps}(F''(\p_n+m_0)-F''(\p+m_0))\xi .
\end{array}$$
Then, by using Lemma \ref{DiferenciaF}, the  $H^{-1}_*$ norm of each term on the right side of the previous expression can be bounded by $C\Vert \p_n-\p\Vert_{L^2}$, $C\Vert  \p_n-\p\Vert_{H^1}$ or $C\Vert \p_n-\p\Vert_{H^2}$ multiplied by $\Vert\xi\Vert_3$. In fact, let us precise the bound of some of these terms:
$$
\begin{array}{l}
\Vert (F'''(\p_n+m_0)-F'''(\p+m_0))\Delta\p\;\xi\Vert_{H^{-1}_*}\leq C
\vert F'''(\p_n+m_0)-F'''(\p+m_0)\vert_3\vert\Delta\p\vert_2\vert\xi\vert_\infty\\
\leq C(K)\vert \p_n-\p\vert_3\vert\xi\vert_\infty\leq
C(K)\Vert \p_n-\p\Vert_1\Vert\xi\Vert_2,
\end{array}$$
$$\begin{array}{l}
\Vert F^{iv)}(\p_n+m_0)(\grad\p_n+\grad\p)\cdot(\grad\p_n-\grad\p)\,
\xi\Vert_{H^{-1}_*}\\
\leq C
\vert F^{iv)}(\p_n+m_0)\vert_\infty\vert\grad\p_n+\grad\p\vert_3\vert\grad\p_n-\grad\p\vert_2\vert\xi\vert_\infty
\leq C(K)\Vert \p_n-\p\Vert_1\Vert\xi\Vert_2,
\end{array}$$
$$
\Vert F'''(\p_n+m_0)\Delta(\p_n-\p) \xi \Vert_{H^{-1}_*}
\le \vert F'''(\p_n+m_0) \vert_3  \vert \Delta(\p_n-\p) \vert_2 \vert \xi \vert_\infty
\\
\le C(K)\Vert \p_n-\p\Vert_2\Vert\xi\Vert_2, 
$$
$$\begin{array}{l}
\Vert (\mathcal{A}(\p_n+m_0)-\mathcal{A}(\p+m_0))\Delta\xi\Vert_{H^{-1}_*}
\leq C \vert \mathcal{A}(\p_n+m_0)-\mathcal{A}(\p+m_0)\vert\, \vert\Delta\xi\vert_2\\
\leq C(K)\Vert \p_n-\p\Vert_1\Vert\xi\Vert_2,
\end{array}$$
Thus, 
$$\begin{array}{c}
\displaystyle
\Vert\mathcal{M}'(\p_n)-\mathcal{M}'(\p)\Vert_{\mathcal{L}(H_1^3;H^{-1}_*)} =
\sup_{\xi\in H_1^3\setminus \{0\}}
\fra{ \Vert\mathcal{M}'(\p_n)(\xi)-\mathcal{M}'(\p)(\xi)\Vert_{H^{-1}_*}}{\Vert \xi\Vert_{H_1^3}}
\\ 
\noalign{\medskip}
\displaystyle
\quad =  \sup_{\xi\in H_1^3\setminus \{0\}}
\fra{ \Vert \overline{G}'(\p_n)(\xi)-\overline{G}'(\p)(\xi)\Vert_{H^{-1}_*}}{\Vert \xi\Vert_{3}}
\leq C \Vert \p_n-\p\Vert_{H^2}.
\end{array}$$
Therefore, $\Vert\mathcal{M}'(\p_n)-\mathcal{M}'(\p)\Vert_{\mathcal{L}(H_1^3;H^{-1}_*)} \to 0$ as $n\to \infty$ if $\p_n\to \p$ in $H^3$ (even if $\p_n\to \p$ in $H^2$), hence the continuity of the operator ${\cal R}$ is proved. 

Finally, by applying Theorem~\ref{huang},  there exists $\beta>0$, $C>0$ and $\theta\in (0,1/2)$ such that for any $\p\in H^3_1(\Om)$ with  $\Vert\p-{\p}_*\Vert_3\leq \beta$, one has
$$
\vert \overline\mathcal{E}_b(\p)-\overline\mathcal{E}_b({\p}_*)\vert^{1-\theta}\leq 
C\,\Vert\mathcal{E}'(\p)\Vert_{H^{-1}_*}
=C\,\Vert z(\p)\Vert_{H^{-1}_*}
$$
and (\ref{conclu}) holds.
\medskip

\noindent{\bf Step 2: }{\it (Relaxing the local approximation $\Vert\p-{\p}_*\Vert_3\leq \beta$ by  $\Vert\p\Vert_2\leq K$, $\Vert\p-{\p}_*\Vert_1\leq \beta_1$ and $\vert \overline{\mathcal{E}} _b(\p)-\overline{\mathcal{E}} _b({\p}_*)\vert\leq\beta_2$). If $\p\in H^3_1(\Om)$ with $\Vert\p\Vert_2\leq K$,  $\Vert\p-{\p}_*\Vert_1\leq \beta_1$ and $\vert \overline\mathcal{E}_b(\p)-\overline\mathcal{E}_b({\p}_*)\vert\leq\beta_2$, then there exits $C>0$ and $\theta\in(0,1/2]$ (depending on ${\p}_*$, $K$, $\beta_1$ and $\beta_2$) such that (\ref{conclu})  holds.}
 \medskip
 
In this step,  we follow  Lemma 4.4 of \cite{SegattiWu} but imposing now the ``proximity'' condition between $\p$ and ${\p}_*$ only in the $H^1$-norm instead of in the $H^2$-norm as in \cite{SegattiWu}.

Since $z(\p_*)=0$, 
$$\Vert  z(\p)\Vert_{H^{-1}_*}=\Vert  z(\p)-z(\p_*)\Vert_{H^{-1}_*}
=\Vert \eps A(\p-{\p}_*)+\overline{G}(\p)-\overline{G}(\p_*)\Vert_{H^{-1}_*}.$$
From (\ref{norm-A}),   
$\Vert \eps A(\p-{\p}_*)\Vert_{H^{-1}_*}=\vert \eps\grad\Delta(\p-{\p}_*)\vert_2$ and from (\ref{n4wA}), there is a constant $M>0$, such that 
$
\Vert\p-{\p}_*\Vert_3\leq M\vert\grad\Delta(\p-{\p}_*)\vert_2
$, hence,
 \begin{equation}\label{cotaLS0}
 \begin{array}{l}
 \Vert  z(\p)\Vert_{H^{-1}_*}\geq \eps \vert\grad\Delta(\p-{\p}_*)\vert_2 -\Vert\overline{G}(\p)-\overline{G}({\p}_*)\Vert_{H^{-1}_*}\\
 \phantom{ \Vert  z(\p)\Vert_{H^{-1}_*}}
 \geq 
 \fra{\eps}{M}\Vert\p-{\p}_*\Vert_3-C\vert\overline{G}(\p)-\overline{G}({\p}_*)\vert_2.
 \end{array}
\end{equation}
 On the other hand, since $\Vert\p\Vert_2\leq K$, we can use the Lemma \ref{DiferenciaG} to bound $\vert \overline{G}(\p)-\overline{G}({\p}_*)\vert_{2}$ as
$$
\begin{array}{l}\label{difG}
\vert \overline{G}(\p)-\overline{G}({\p}_*)\vert_2\leq
\vert G(\p+m_0)-G({\p}_*+m_0)\vert_2+\vert \langle G(\p+m_0)\rangle-\langle G({\p}_*+m_0)\rangle\vert_2
\\\phantom{\vert \overline{G}(\p)-\overline{G}({\p}_*)\vert_2}
\leq C(K)\,\Vert\p-{\p}_*\Vert_2.
\end{array}$$
In particular, interpolating $H^2_1$ between $H^1_*$ and $H^3_1$, we obtain that
$$C\vert \overline{G}(\p)-\overline{G}({\p}_*)\vert_{2}\leq C_1(K)\,\Vert\p-{\p}_*\Vert_1^{1/2}\Vert\p-{\p}_*\Vert_3^{1/2}\leq 
\fra{\eps}{2M}\Vert\p-{\p}_*\Vert_3+\fra{M}{2\eps}C_1(K)^2\Vert\p-{\p}_*\Vert_1.$$

Let $\beta>0$, $\theta\in (0,1/2)$ given in Step 1, by choosing $\beta_1>0$ and $\beta_2>0$, both sufficiently small,  such that 
$$\fra{M}{2\eps}C_1(K)^2\beta_1\leq \fra{\eps\beta}{4M}\qquad\hbox{and}\qquad \beta_2^{1-\theta}\leq \fra{\eps\beta}{4M},$$
then, for any $\p\in H^3_1(\Om)$ satisfying  $\Vert\p\Vert_2\leq K$, $\Vert\p-{\p}_*\Vert_1\leq\beta_1$ and $\vert \overline{\mathcal{E}} _b(\p)-\overline{\mathcal{E}} _b({\p}_*)\vert\leq \beta_2$, we have  
\begin{equation}\label{cotaLS1}
C\vert \overline{G}(\p)-\overline{G}({\p}_*)\vert_{2}\leq 
\fra{\eps}{2M}\Vert\p-{\p}_*\Vert_3+\fra{\eps\beta}{4M}
\end{equation}
and
\begin{equation}\label{cotaLS2}
\vert \overline\mathcal{E}_b(\p)-\overline\mathcal{E}_b({\p}_*)\vert^{1-\theta}\leq \fra{\eps\beta}{4M}.
\end{equation}
There are only two possibilities: 
either $\Vert\p-{\p}_*\Vert_3\le \beta$ and then (\ref{conclu}) holds by using Step 1,  or $\Vert\p-{\p}_*\Vert_3>\beta$. In this latter case, from (\ref{cotaLS0}) and (\ref{cotaLS1})
$$\begin{array}{l}
 \Vert  z(\p)\Vert_{H^{-1}_*}\geq\fra{\eps}{2M}\Vert\p-{\p}_*\Vert_3- \fra{\eps\beta}{4M}
>\fra{\eps\beta}{2M}-\fra{\eps\beta}{4M}=\fra{\eps\beta}{4M}.
\end{array}$$
On the other hand, from (\ref{cotaLS2}), 
$$
\fra{\eps\beta}{4M} \ge 
 \vert \overline\mathcal{E}_b(\p)-\overline\mathcal{E}_b({\p}_*)\vert^{1-\theta} 
 $$
hence (\ref{conclu})  holds.
 \end{proof}
\section{Weak Solutions} \label{WS}
We define the {\it total free-energy} as
$
\overline\mathcal{E}(\u,\p)=\mathcal{E}_k(\u) + \lambda\,\overline\mathcal{E}_b(\p),
$
 where $\displaystyle \mathcal{E}_k(\u)=\frac{1}{2} \jnt_\Omega |\u|^2$ is the kinetic energy, and  $\overline\mathcal{E}_b(\p)$ is the bending energy defined in  (\ref{identenergy1}). That is,
 \begin{equation}\label{totalenergy}
\overline\mathcal{E}(\u,\p)=\frac{1}{2} \jnt_\Omega |\u|^2+\fra{\lambda}{2\eps}\jnt_\Omega ( -\eps\Delta \p+\fra 1\eps F'(\p+m_0))^2\, d\x+\fra{M}{2}(\mathcal{A}(\p+m_0)-\alpha)^2.
\end{equation}
\begin{definition}\label{weakdef} 
 Let  $\u_0\in \H$ and $\p_0=\phi_0-m_0\in H^2_1$, we say that $(\u,\p,z)$ is a global
weak solution of $(\ref{P1s})$-$(\ref{cis})$ in $(0,+\infty)$, if 
\begin{equation}\label{defi}
\begin{array}{c}\u\in L^\infty(0,+\infty;\H) \cap L^2(0,+\infty;\V)
,\\
 \p\in L^\infty(0,+\infty;H^2_1)\cap L^2_{\rm loc}([0,+\infty);H^3_1),
 \quad z\in L^2(0,+\infty;H^1_*),
\end{array}
\end{equation}
 satisfying the variational formulation a.e.~$t\in (0,+\infty)$:
\begin{eqnarray}\label{weak-form-1s}
   \langle \partial_t \u,\overline{\u} \rangle +
  ((\u\cdot\grad)\u,\overline{\u})
  +\nu(\grad \u,\grad \overline{\u})
  -\lambda(z \grad\p, \overline{\u}) = 0,\quad \forall \, \overline{\u} \in \V
\\\label{weak-form-2s}
  \langle\partial_t \p,\overline{z}\rangle
 + (\u\cdot\grad\p,\overline{z})+\gamma(\grad z,\grad\overline{z})=0,\quad \forall \, \overline{z}\in H^1_*
 \\\label{weak-form-3s}
 -\eps (\grad\Delta\p,\grad\overline{\p})+(\overline{G}(\p),\overline{\p})-(z,\overline{\p})=0,\quad \forall \, \overline{\p}\in H^1_*
 \end{eqnarray}
the initial conditions (\ref{cis}), and  the energy inequality (in integral version)
\begin{equation}\label{energyeqint}
\overline\mathcal{E}(\u(t_1),\p(t_1))-\overline\mathcal{E}(\u(t_0),\p(t_0))+ \jnt_{t_0}^{t_1}(\nu|\grad{\u}(s)|^2_2 +
\lambda\gamma|\grad z(s)|^2_2)\;ds  \leq 0,
\end{equation}
for a.e. $t_1,t_0: t_1\geq t_0\ge 0$. 
 \end{definition}
 
 \
 
 Energy inequality (\ref{energyeqint}) shows the dissipative character of the model with respect to the total free-energy $\overline\mathcal{E}(\u(t),\p(t))$.
 
 Observe that, using the operator $A$ defined in (\ref{Aoperator}), then  (\ref{weak-form-3s}) can be rewritten as  
 \begin{equation} \label{weak-form-3s-a}
z(t)=\eps A\p(t)+\overline{G}(\p(t)).  
\end{equation}

 On the other hand, from (\ref{defi}), (\ref{weak-form-1s}) and (\ref{weak-form-2s}),  one can deduce  that 
 $$\partial_t \u\in L^{4/3}_{\rm loc}([0,+\infty);\V\,')\qquad\mbox{and}\qquad\partial_t\p\in L^2_{\rm loc}([0,+\infty); H^{-1}_*),$$
  hence, the following time-continuity can be deduced (see Lemma 1.4 of \cite{Temam}), 
  $$ \u\in C([0,+\infty);\V\,')\cap C_w([0,+\infty);\H),$$
  $$ \p\in C([0,+\infty);H^{-1}_*)\cap C_w([0,+\infty);H_1^2).
  $$
 Here, the weak continuity $\u\in C_w([0,+\infty);\H)$ means that $\u(t)\in \H$ for all $t\in [0,+\infty)$ and $\u(s)\to \u(t)$ weak in $\H$ as $s\to t$. 
 In particular, the initial conditions (\ref{cis}) have a sense. 
 
 \
 
 The following lemma gives two improved energy inequalities.
 \begin{lemma} \label{le:12}
 Given $(\u,\p,z)$  a global weak solution of $(\ref{P1s})$-$(\ref{cis})$ in $(0,+\infty)$, there exists a regularized energy function $\widetilde{\mathcal{E}}=\widetilde{\mathcal{E}} (t)\in \mathbb{R}$ defined for all $t\geq 0$, which satisfies the following integral inequality for all $t_1, t_0$ with $ t_1\geq t_0\geq 0$:
%
 \begin{equation} \label{energyeqint-b}
\widetilde{\mathcal{E}}(t_1)-\widetilde{\mathcal{E}}(t_0)+ \jnt_{t_0}^{t_1}(\nu|\grad{\u}(s)|^2_2 +
\lambda\gamma|\grad z(s)|^2_2)\;ds  \leq 0,
\end{equation} 
 and the following differential inequality for a.e.~$t\ge 0$:
 \begin{equation}\label{energy-eqs}
  \displaystyle \frac{d}{dt} \widetilde{\mathcal{E}}(t)+ \nu|\grad{\u(t)}|^2_2 +\lambda
\gamma|\grad z(t)|^2_2  \leq 0,\quad \hbox{a.e.~$t\ge 0$.}
\end{equation}
 \end{lemma}
 \begin{proof}
 From the weak continuity  $ \u\in C_w([0,+\infty);\H)$ and $ \p\in C_w([0,+\infty);H_1^2)$, in particular the energy evaluated in the trajectory $\overline\mathcal{E}(\u(t),\p(t))$ exits for all $t$. Moreover, since the inequality (\ref{energyeqint}) is satisfied for all $t_0, t_1\in [0,+\infty)\backslash N$, where $N$ is a set of null Lebesgue measure, then $\overline\mathcal{E}(\u(t),\p(t))$ is a real decreasing function in $ [0,+\infty)\backslash N$, and $\overline\mathcal{E}(\u(t),\p(t)) \in L^\infty(0,+\infty)$.  Therefore, we can define a regularized   function $\widetilde{\mathcal{E}}$ related to $\overline\mathcal{E}$ for all $t\ge 0$ as:
$$
\widetilde{\mathcal{E}}(0):=\overline\mathcal{E}(u_0,\p_0), \qquad \widetilde{\mathcal{E}}(t):=\lim_{\stackrel{\:\:\scriptstyle s\to t^-} {\scriptstyle s \in [0,+\infty)\setminus N}}\overline\mathcal{E}(\u(s),\p(s)).
$$
Then, this function $\widetilde{\mathcal{E}}(t)$ is ``continuous from the left"  and is decreasing for all $t\ge 0$. Indeed, for any $t_1, t_2\in  [0,+\infty)$, for instance $t_1< t_2$, we can choose sequences $\{s^1_n\}, \{s^2_n\} \subset [0,+\infty)\backslash N$ such that $s_n^1\to t^-_1$, $s_n^2\to t^-_2$ and,   $s^1_n \leq s^2_n$ for all  $n\geq n_0$. 
Since $s_n^1$ and  $s_n^2$ are not in $N$, we know that $\overline\mathcal{E}(\u(s_n^1),\p(s_n^1))\geq \overline\mathcal{E}(\u(s_n^2),\p(s_n^2))$. By taking limit as $s_n^1\to t^-_1$ and $s_n^2\to t^-_2$, we obtain that $\widetilde{\mathcal{E}}(t_1)\geq\widetilde{\mathcal{E}}(t_2)$.

Since $\widetilde{\mathcal{E}}(t)$ is decreasing for all $t\in [0,+\infty)$, it is derivable (and continuous) almost everywhere $t\in (0,+\infty)$ (see for example Theorem~3 of \cite{Royden}{ p.~100}).


On the other hand, since the inequality (\ref{energyeqint}) is satisfied for all $t_0, t_1 \in [0,+\infty)\setminus N$, given any $t_0\leq t_1$, we can take $\delta_n>0$ and $\eta_n>0$ such that $t_0-\delta_n,\:t_1-\eta_n \not\in N$ and $\delta_n\to 0$ and $\eta_n\to 0$. Then  
$$
\overline\mathcal{E}(\u(t_1-\delta_n),\p(t_1-\delta_n))
- \overline\mathcal{E}(\u(t_0-\delta_n),\p(t_0-\delta_n))
+ \jnt_{t_0-\delta_n}^{t_1-\eta_n}(\nu|\grad{\u}(s)|^2_2 +
\lambda\gamma|\grad z(s)|^2_2)\;ds  \leq 0.
$$
By taking $\delta_n\to 0$ and $\eta_n\to 0$, and using the definition of $\widetilde{\mathcal{E}}(t)$, we obtain (\ref{energyeqint-b}) for any $t_0\leq t_1$.

In particular,   
by choosing $t_0=t$ and $t_1=t+h$ in (\ref{energyeqint-b}), we obtain
 \begin{equation}\label{interm}
\fra{\widetilde{\mathcal{E}}(t+h)-\widetilde{\mathcal{E}}(t)}{h} +\fra{1}{h} \jnt_{t}^{t+h}(\nu|\grad{\u}|^2_2 +
\lambda\gamma|\grad z|^2_2)\;ds  \leq 0,\quad \forall\, t,h\ge 0.
\end{equation}
Observe that, since the map $t\in [0,+\infty)\to \nu|\grad{\u(t)}|^2_2 +\lambda
\gamma|\grad z(t)|^2_2\in \mathbb{R}$ belongs to $L^1(0,+\infty)$, then
$$
\lim_{h\to 0}\fra{1}{h} \jnt_{t}^{t+h}(\nu|\grad{\u}|^2_2 +
\lambda\gamma|\grad z|^2_2)\;ds  = \nu|\grad{\u(t)}|^2_2 +\lambda
\gamma|\grad z(t)|^2_2 \qquad \mbox{a.e.}~t\ge 0.
$$
By taking $h\to 0$ in (\ref{interm}), we obtain (\ref{energy-eqs}) a.e.~$t\ge 0$.
 \end{proof}

 \
 
 \begin{remark}
The fact that the energy is decreasing is essential in Lemma~\ref{le:12}. For general problems with nonzero external forces, the thesis of this Lemma is not clear.
 \end{remark}
 
%

\subsection{Formal energy equality and large-time weak regularity}\label{Se:formal}
We will arrive at the energy equality related to (\ref{energy-eqs}) in a formal manner (assuming a sufficiently regular solution). A rigorous demonstration of existence of weak solutions (in particular, satisfying (\ref{energyeqint})) will be obtained in Section~\ref{Sec:Galerkin} via a Galerkin method, in the same way as was done in \cite[pp.~71-73]{ConstantinFoias} for the Navier-Stokes problem. By assuming that  $(\u,\p, z)$ is a sufficiently regular solution of (\ref{P1s})-(\ref{cis}), and taking $\overline{\u}=\u$, $\overline{z}=z$ and $\overline{\p}=\partial_t\p$ as test
function in (\ref{weak-form-1s}), (\ref{weak-form-2s}) and (\ref{weak-form-3s}) respectively,  one arrives at the equalities:
  $$
\fra{1}{2}\fra{d}{dt}|\u|^2_2+\nu|\grad\u|_2^2-\lambda(z\, \grad \p,\u)=0,
$$
$$
  (\partial_t\p,z)+(\u\cdot\grad\p,z)+\gamma |\grad z|_2^2=0,
  $$
  $$
 \eps \fra{d}{dt}\fra{1}{2} |\Delta\p|_2^2+
 ( \overline{G}(\p),\partial_t\p) -(z,\partial_t\p)=0. 
  $$
Adding the first equality plus the second and third ones multiplied by $\lambda$, the term $(z,\partial_t\p)$ cancels, as well as the nonlinear convective term $(\u\cdot\grad\p,z)$ with the elastic term $-(z\, \grad \p,\u)$, arriving at 
 \begin{equation}\label{energy-eqs0}
\fra{1}{2}\fra{d}{dt}(|\u|^2_2+\lambda\eps|\Delta\p|_2^2)+ \lambda ( \overline{G}(\p),\partial_t\p)+\nu|\grad{\u(t)}|^2_2 +\lambda
\gamma|\grad z(t)|^2_2  = 0.
\end{equation}
Since $\fra{\delta \overline\mathcal{E}_b(\psi)}{\delta \psi}=z=\eps\Delta^2\p+\overline{G}(\p)$, in particular 
$$\begin{array}{c}
\displaystyle \frac{d}{dt} \overline\mathcal{E}_b(\p(t))=\Big\langle \fra{\delta \overline\mathcal{E}_b(\psi)}{\delta \p}, \partial_t\psi\Big\rangle=
 ( z,\partial_t\p)
 =\eps \fra{1}{2} \fra{d}{dt} |\Delta\p|_2^2+
 ( \overline{G}(\p),\partial_t\p).
 \end{array}$$
 Then, equality (\ref{energy-eqs0}) can be rewriten as the following {\it energy equality} (which is the equality version of (\ref{energy-eqs})):
 \begin{equation}\label{energy-eqs2}
  \displaystyle \frac{d}{dt} \overline\mathcal{E}(\u(t),\p(t))+ \nu|\grad{\u(t)}|^2_2 +\lambda
\gamma|\grad z(t)|^2_2  = 0,
\end{equation}
 From (\ref{energy-eqs2}), assuming the initial regularity $(\u_0,\p_0)$ in $\H\times H^2_1$ hence $\overline\mathcal{E}(\u(0),\p(0))$ is finite, we obtain that $\u \in L^\infty(0,+\infty;\textit{\textbf{H}})\cap L^2(0,+\infty;\textit{\textbf{V}})$,  $z \in L^2(0,+\infty; H^1_*)$, and $\overline\mathcal{E}_b(\p(t))\in L^\infty(0,+\infty)   $. In particular, 
$$\jnt_\Omega ( -\eps\Delta \p+\fra 1\eps F'(\p + m_0))^2\, d\x
\quad \hbox{and} \quad \mathcal{A}(\p)=\jnt_\Omega\left(\fra{\eps}{2}\vert \grad\p\vert^2+\fra{1}{\eps}F(\p+m_0)\right)\, d\x$$ belong to $ L^\infty(0,+\infty)$.

 From $\mathcal{A}(\p)\in L^\infty(0,+\infty)$ one has that $\grad \p\in  L^\infty(0,+\infty;\L^2)$ and, since $\displaystyle \int_\Omega \p =0$, also $\p \in L^\infty(0,+\infty;H^1_*)$. Secondly, from $\jnt_\Omega ( -\eps\Delta \p+\fra 1\eps F'(\p + m_0))^2 \in L^\infty(0,+\infty)$,  we have that $\p \in L^\infty(0,+\infty;H^2_1)$.

In summary, the following global in time  ``weak''  regularity (in the time interval $(0,+\infty)$) hold:
\begin{equation}\label{estim1}
\u \in L^\infty(0,+\infty;\textit{\textbf{H}})\cap L^2(0,+\infty;\textit{\textbf{V}}),
\quad \p \in L^\infty(0,+\infty;H^2_1)
,   \quad z \in L^2(0,+\infty; H^1_*).
\end{equation}
By applying Remark \ref{remark6} to $ \overline{G}(\p(t))$ defined in (\ref{Gbarra}), since $\Vert\p(t)+m_0\Vert_2\leq K$, we obtain
\begin{equation}\label{L}
\vert \overline{G}(\p(t))\vert_2\leq \vert G(\p(t)+m_0)\vert_2+\vert \langle G(\p(t)+m_0)\rangle \vert_2 \leq C.
\end{equation}
From (\ref{n4wA}), (\ref{norm-A}) and (\ref{weak-form-3s-a}), we obtain
\begin{equation}\label{H3-estim}
\Vert\p\Vert_3\le C\, \vert\grad\Delta\p\vert_{2}
= C\, \Vert A\p\Vert_{H^{-1}_*}
\leq C (\vert z\vert_{2}+ \vert \overline{G}(\p) \vert_{2})
\leq C(1+\vert z\vert_{2})
\end{equation}
From this estimate and the regularity of $z$ given in (\ref{estim1}), we have
\begin{equation}\label{phi3}
\p\in L^2_{loc}([0,+\infty); H^3_1).
\end{equation}
\subsection{Existence of weak solutions satisfying energy inequality (\ref{energyeqint})}
\label{Sec:Galerkin}
For instance,  fixed the initial datum $(\u_0,\p_0)\in\H\times H^2_1$, the existence of weak solutions of (\ref{P1s})-(\ref{cis}) in $(0,+\infty)$,  furnished by a limit of an adequate Galerkin approximate solutions  can be proved analogously to \cite{smectic} or \cite{nu}. Moreover, these Galerkin  solutions satisfy the corresponding energy equality  (\ref{energy-eqs2}) which  suffices to prove  rigorously the previous estimates (\ref{estim1}) and (\ref{phi3}).
Finally, following similar arguments to those used in \cite{ConstantinFoias} for the Navier-Stokes equations, these weak solutions will 
 satisfy the energy inequality (\ref{energyeqint}). 
 
 Indeed, let $\{\w_i\}_n\geq 1$ and $\{\phi_i\}_n\geq 1$ ``special'' basis of
$\V$ and $H^2_1$,
  respectively, formed by
eigenfunctions of the Stokes problem
$$
\w_i\in \V \quad\mbox{ such that }\quad  S \w_i=\lambda_i \w_i,\quad\mbox{ with }\Vert \w_i\Vert_{L^2}=1,
\quad\lambda_i\nearrow+\infty\qquad
$$
and of the Poisson-Neumann problem
$$\phi_i\in H^2_1 \quad\mbox{ such that }\quad -\Delta\phi_i=\mu_i\,\phi_i,\quad\mbox{ with }\Vert
\phi_i\Vert_{L^2}=1, \quad\mu_i\nearrow+\infty.
$$
Here, we consider the Stokes operator $S$ defined as $S\w\in \V$ such that  
$(\grad (S\w), \grad \v)=(\w,\v)$ for all $\v\in \V$. 

Let $\V^m$ and $W^m$ be the finite-dimensional subspaces spanned
by $\{\w_1,\w_2,\dots,\w_n\}$ and $\{\phi_1,\phi_2,\dots,\phi_n\}$
respectively. Note that if $\phi\in W^m$ then also $\Delta\phi\in W^m$ and  $\Delta^2\phi\in W^m$.

For each $m\geq 1$, we say that $(\u_m,\p_m)$ is a Galerkin 
solution, if $\u_m:[0,+\infty)\mapsto \textit{\textbf{V}}^m$ and 
 $\p_m:[0,+\infty)\mapsto W^m$, and satisfy
\begin{equation}\label{pm}
\left\{
\begin{array}{l}
   \langle \partial_t \u_m, \overline{\u}_m \rangle +
  ((\u_m\cdot\grad)\u_m, \overline{\u}_m)
  +\nu(S \u_m, \overline{\u}_m)\\\qquad
  -\lambda(z_m \grad\p_m, \overline{\u}_m) = 0,\quad \forall \, \overline{\u}_m \in \V^m,\qquad \hbox{a.e.~$t>0$},
\\
(\partial_t \p_m, \overline{z}_m)+ (\u_m\cdot\grad\p_m, \overline{z}_m)+\gamma(\grad z_m,\grad \overline{z}_m)=0,\quad \forall \, \overline{z}_m \in W^m, \quad\mbox{$\forall\, t> 0$},
 \\
 \noalign{\medskip}
  \u_m(0)=\u_{0m}:=P_m(\u_0),
   \quad \p_m(0)=\p_{0m}=Q_m(\p_0) \quad
\mbox{in }\Om,
\end{array}
\right.
\end{equation}
where $z_m:[0,+\infty)\mapsto W^m$ is defined as
$z_m:= Q_m(\eps \Delta^2\p_m+\overline{G}(\p_m))$. 
Here, 
$P_m:\L^2\mapsto \V^m$ denotes the 
projection from $\L^2$ onto $\V^m$ and 
$Q_m:L^2\mapsto W^m$  the projection from $L^2$ onto
$W^m$. Since $\Delta^2\p_m\in  W^m$ then 
\begin{equation} \label{z_m}
z_m
= \eps \Delta^2\p_m+Q_m(\overline{G}(\p_m))
= \eps A\p_m+Q_m(\overline{G}(\p_m)),\quad \hbox{in $\Omega\times (0,+\infty)$.}
\end{equation}
Since $(\partial_t\u_m, \partial_t\p_m)\in \V^m\times W^m$ and $(S\u_m,-\Delta z_m)\in \V^m\times W^m$,  the variational formulation $(\ref{pm})$ yields to the point-wise equalities
\begin{equation}\label{u_m-a}
\partial_t \u_m + \nu S\u_m + P_m((\u_m\cdot\grad)\u_m  -\lambda\, z_m \grad\p_m)  =0,
\quad \hbox{in $\Omega\times (0,+\infty)$.}
\end{equation}
\begin{equation}\label{psi_m-a}
\partial_t \p_m -\gamma \Delta z_m +Q_m(\u_m\cdot\grad\p_m) =0,
\quad \hbox{in $\Omega\times (0,+\infty)$.}
\end{equation}

The existence and uniqueness of local in time solution  of
(\ref{pm})-(\ref{z_m}) in $\Om\times (0,T)$, for any $T>0$ sufficiently small, can be proved 
using the Leray-Schauder's Theorem in the same way as the Theorem 5.2 of \cite{smectic}. By following the argument given in Subsection \ref{Se:formal}, 
the energy equality 
\begin{equation}\label{energyeqintm}
\overline\mathcal{E}(\u_m(t),\p_m(t))-\overline\mathcal{E}(\u_m(t_0),\p_m(t_0))+ \jnt_{t_0}^t(\nu|\grad{\u_m}|^2_2 +
\lambda\gamma|\grad z_m|^2_2)\;d\tau  = 0
\end{equation}
holds for all $t,t_0: t\geq t_0\ge 0$. 
Therefore,  the  solution  of (\ref{pm})-(\ref{z_m}) can be extended to the whole time interval $(0,+\infty)$, and  the following uniform in time estimates (independent of $m$) can be obtained as in Subsection \ref{Se:formal}:
$$\u_m \mbox{ in } L^\infty(0,+\infty;\H)\cap L^2(0,+\infty;\V),\qquad\p_m \mbox{ in } L^\infty(0,+\infty;H^2_1), \qquad z_m \mbox{ in } L^2(0,+\infty;H^1_*).$$

By applying these estimates in (\ref{u_m-a})-(\ref{psi_m-a}) one has the estimates for the time derivatives
 $$\partial_t \u_m \mbox{ in } L^{4/3}_{\rm loc}([0,+\infty);\V\,')\qquad\mbox{and}\qquad\partial_t\p_m \mbox{ in } L^2_{\rm loc}([0,+\infty); H^{-1}_*).$$
 By following the argument given in Subsection \ref{Se:formal} to prove (\ref{L}) and  (\ref{phi3}), one has the estimates
 $$
  \overline{G}(\p_m) \mbox{ in } L^{\infty}((0,+\infty)\times \Om)
  \qquad\mbox{and}\qquad
 \p_m\mbox{ in } L^2_{loc}([0,+\infty); H^3_1).
 $$
 By using compactness Lemma \ref{simon} for the spaces $\V \subset \H \subset \V'$ and $H^2_1\subset W^{1,p}_*(\Om)\subset H^{-1}_*$ with $p<6$, and considering the continuous embedding $W^{1,p}(\Om)\subset C(\overline\Om)$ ($p>3$), we have that  $\u_m$ is  
relatively compact in $L^2(0,T;\H)$, $\p_m$ is  
relatively compact in $C([0,T]\times \overline\Om)$ and  $\grad\p_m$ in $C([0,T];L^p(\Om))$, with $p<6$ for any $T>0$. In particular,  from Lemma \ref{DiferenciaF}, $F(\p_m)$  (and $F'(\p_m)$, $F''(\p_m)$) is  
relatively compact in $C([0,T]\times \overline\Om)$. Finally, by using again  Lemma \ref{simon} for the spaces $H^3_1\subset H^2_1\subset H^{-1}_*$ then $\p_m$ is  
relatively compact in $L^2(0,T;H^2_1)$. 

In order to prove the existence of a weak solution $(\u,\p,z)$ of the variational  problem
$(\ref{weak-form-1s})$-$(\ref{weak-form-3s})$ a.e.~$t\in (0,+\infty)$, obtained as the limit of a (subsequence) of $(\u_m,\p_m,z_m)$, it suffices to pass to the limit in  (\ref{z_m})-(\ref{psi_m-a}), which will be possible due to the previous estimates and compactness.
For this, given any $T>0$ and $(\overline\u,\overline{w})\in \V\times H^2_1$ we take $(\overline \u_m,\overline{w}_m)\in \V^m \times W^m$ with $(\overline \u_m,\overline{w}_m)\rightarrow (\overline \u,\overline{w})$ in $\V\times H^2_1$ strongly. We rewrite (\ref{z_m})  as 
\begin{equation} \label{z_m-a}
(z_m,\overline{w}_m)=\eps \langle  A\p_m,\overline{w}_m\rangle+(\overline{G}(\p_m),\overline{w}_m), \quad\mbox{$\forall\, t\in [0,T]$}.
\end{equation}
To prove the weak convergence 
\begin{equation} \label{conv-G}
\overline{G}(\p_m) \to \overline{G}(\p)\quad \hbox{weakly in $L^2((0,T)\times \Om)$,}
\end{equation}
  we analyze for example, the term $F''(\p_m+m_0)\Delta\p_m$. 
Since
 $F''(\p_m+m_0)$ converges to $F''(\p+m_0)$
 in $C([0,T]\times \overline\Om)$ and $ \Delta\p_m$ converges weakly to $ \Delta\p$ in $L^2((0,T)\times \Om)$, hence $F''(\p_m+m_0)\Delta\p\rightarrow F''(\p+m_0)\Delta\p$ weakly in $ L^2((0,T)\times \Om)$. The proof for the rest of the terms of $\overline{G}(\p_m)$  is easier, and then, one has the  convergence of (\ref{z_m-a}) towards the limit equation~(\ref{weak-form-3s}).
  
For example, the convergence of the term $(z_m \grad\p_m, \overline{\u}_m) $ of $(\ref{pm})_1$ towards $(z \grad\p, \overline{\u}) $ is deduced from the strong convergence of $\grad\p_m$ in $L^{\infty}(0,T;L^2( \Om))$ and the weak convergence of $z_m$ in $L^2(0,T;H^1_*)$. The convergence of the  rest of  terms of $(\ref{pm})_1$ is rather standard in the Navier-Stokes setting, see \cite{ConstantinFoias}, hence (\ref{weak-form-1s}) holds. 

Finally, the limit of $(\ref{pm})_2$ is standard, arriving at (\ref{weak-form-2s}).

In order  to prove that the weak solutions furnished by previous Galerkin method satisfy  the energy inequality (\ref{energyeqint}), we follow similar arguments used in \cite{ConstantinFoias} for the Navier-Stokes equations. 

Firstly, we are going to prove that 
\begin{equation}\label{en1}
\overline\mathcal{E}(\u_m(t),\p_m(t))\rightarrow \overline\mathcal{E}(\u(t),\p(t))\quad\mbox{a.e. } t\geq 0.
\end{equation}
In fact, to show the convergence of the most difficult term $\int_\Om {\mu}(\p_m+m_0)^2 d\x$ in (\ref{en1}), it suffices to prove that
\begin{equation}\label{en2}
{\mu}(\p_m+m_0)\rightarrow {\mu}(\p+m_0)\quad\mbox{strongly in  } L^2(0,T;L^2),\quad \forall\, T>0.
\end{equation}
For this, we use that 
$\p_m$ is bounded in $L^2(0,T;H^3_1)$,
 and in particular  $\Delta \p_m$ is bounded in $L^2(0,T;H^1_*)$. Moreover,  $\grad F'(\p_m+m_0)=F''(\p_m+m_0)\grad\p_m$ is bounded in $L^\infty(0,T;L^2)$. Therefore, 
\begin{equation}\label{en3}
{\mu}(\p_m+m_0)\, \mbox{ is bounded in }\, L^2(0,T;H^1).
\end{equation}
On the other hand,  by using  the parabolic equation   (\ref{psi_m-a}),
$$
\begin{array}{l}
\displaystyle
\partial_t{\mu}(\p_m+m_0)=-\eps\Delta\partial_t\p_m + \fra{1}{\eps}F''(\p_m+m_0)\partial_t\p_m
\\
\phantom{\partial_t{\mu}(\p_m(t))}
\displaystyle
=\eps\Delta (Q_m(\u_m\cdot\grad\p_m)) -\gamma  \eps\Delta^2 z_m+\fra{1}{\eps}F''(\p_m+m_0)(\gamma  \Delta z_m-Q_m(\u_m\cdot\grad\p_m)),
\end{array}
$$
Here, the second term at the right hand side has the following sense
$$-\langle \Delta^2 z_m ,\widetilde{\p}\rangle_{(H_1^3)',H_1^3}=(\grad z_m,\grad\Delta\widetilde{\p})_{L^2},\quad \forall\, \widetilde{\p}\in H_1^3.
$$
Hence,  the estimates already obtained for $\u_m$, $\p_m$ and $z_m$ imply that
\begin{equation}\label{en32}
\partial_t{\mu}(\p_m+m_0)\, \mbox{ is bounded in }\, L^2(0,T;(H^{3}_1)'),
\end{equation}
where the dual space $(H^3_1)'$ appears due to the term $\Delta^2 z_m$.

Then, (\ref{en3}) and (\ref{en32}) and Lemma \ref{simon} (for the spaces $H^1_*\subset L^2_*\subset (H^{3}_1)'$) imply (\ref{en2}). On the other hand, since $
\u_m\to\u$ strongly in  $L^2(0,T;\H)$  then
(\ref{en1}) holds.

Secondly, since $ \u_m\to\u$ and $ z_m\to z$, both weakly in $L^2(0,T;H^1)$, we have that
\begin{equation}\label{en4}
\liminf_{m\rightarrow +\infty}\jnt_{t_0}^t(\nu|\grad{\u_m}|^2_2 +
\lambda\gamma|\grad z_m|^2_2)\;d\tau  \geq
\jnt_{t_0}^t(\nu|\grad{\u}|^2_2 +
\lambda\gamma|\grad z|^2_2)\;d\tau  
\quad\mbox{for all } t\geq t_0\geq 0.
\end{equation}

Thirdly, by taking $\liminf$ in (\ref{energyeqintm}), we obtain that for all $t\geq t_0\geq 0$,
\begin{equation}\label{energlim}
\begin{array}{c}
\displaystyle\liminf_{m\rightarrow +\infty}\overline\mathcal{E}(\u_m(t),\p_m(t))+\liminf_{m\rightarrow +\infty} \jnt_{t_0}^t(\nu|\grad{\u_m}|^2_2 +
\lambda\gamma|\grad z_m|^2_2)\;d\tau  
\\\leq 
\displaystyle \limsup_{m\rightarrow +\infty}\overline\mathcal{E}(\u_m(t_0),\p_m(t_0)).
\end{array}\end{equation}

Finally, by using (\ref{en1}) and (\ref{en4}) in (\ref{energlim}), we obtain the energy inequality  (\ref{energyeqint}) for a.e.~$t,t_0:t\geq t_0\geq 0$.
\section{Convergence at infinite time. } \label{sec:time-inftys}
Let $(\u,\p,z)$ be  a weak solution of (\ref{P1s})-(\ref{cis}) in $(0,+\infty)$ associated to an  initial data $(\u_0,\p_0)\in \H\times H^2_1$ (see Definition~\ref{weakdef}).  
From the energy inequality (\ref{energyeqint}), there exists a real number $E_\infty\geq 0$ such that the total energy evaluated in the trajectory $(\u(t),\p(t))$, satisfies 
\begin{equation}\label{asenergys}
\overline{\mathcal{E}}(\u(t),\p(t))\searrow E_\infty\mbox{ in }\mathbb{R} \quad\mbox{ as } t\uparrow +\infty.
\end{equation}
Let us define the  $\omega$-limit set of  this  global weak solution $(\u,\p)$  as follows:
\begin{equation}\label{def-w}
\begin{array}{l}
\omega(\u,\p)=\{(\u_\infty,\p_\infty)\in \H\times H^2_1: \exists \{t_n\}\uparrow+\infty \mbox{ s.t.\ }
\\
\qquad \qquad \qquad
({\u}(t_n),{\p}(t_n))\rightarrow(\u_\infty,\p_\infty) \mbox{ weakly in } \L^2\times H^2_1\}.
 \end{array}
\end{equation}
Let ${\cal S}$ be the set of critical points of the bending energy  defined in (\ref{set-S}), that is
\begin{equation}\label{def-S}
{\cal S}=\{\p\in H^3_1\ : \  -\eps (\nabla\Delta\p,\nabla\widetilde\p)+(\overline{G}(\p),\widetilde\p)=0,\  \forall\,\widetilde\p\in H^1_* \}.
\end{equation}

\begin{theorem} \label{the:first}
Assume that $(\u_0,\p_0)\in \H\times H^2_1$. Fixed $(\u,\p,z)$  a weak solution of (\ref{P1s})-(\ref{cis}) in $(0,+\infty)$,
then $\omega(\u,\p)$ is  nonempty  and $\omega(\u,\p)\subset \{0\}\times {\cal S}$. Moreover, for any  $ \p_\infty \in {\cal S}$ such that $(0,\p_\infty)\in \omega(\u,\p)$, it holds  $\overline{\mathcal{E}}_b(\p_\infty)=E_\infty$. 
In particular, $\u(t)\rightarrow 0$ weakly in $\L^2$ and $\overline{\mathcal{E}}(\u(t),\p(t))\rightarrow\overline{\mathcal{E}} _b(\p_\infty)$ in $\mathbb{R}$ as $t\uparrow +\infty$.
\end{theorem}
\begin{proof}
Since $(\u,\p)\in L^\infty(0,+\infty;\H\times H^2_1)$, there exists $\{t_n\}\uparrow +\infty$ and suitable limit functions $(\u_\infty, \p_\infty)\in \H\times H^2_1$, such that 
\begin{equation}\label{phiinf}
\u(t_n)\rightarrow \u_\infty\mbox{ weakly in }\H,\qquad
\p(t_n)\rightarrow \p_\infty\mbox{ weakly in }H^2_1 .
\end{equation}
We consider the initial and boundary-value problem associated to (\ref{P1s})-(\ref{cis}) restricted on the time interval $[t_n,t_n+1]$ with initial values $\u(t_n)$ and $\p(t_n)$. If we define 
$$
\u_n(s):=\u(s+t_n),\quad\p_n(s):=\p(s+t_n)\quad\mbox{ and }\quad z_n(s):=z(s+t_n)\quad\mbox{ for a.e. }s\in [0,1],$$
then, $(\u_n, \p_n,z_n)$ is a weak solution to the problem (\ref{P1s})-(\ref{cis}) in the time interval $ [0,1]$.
From the energy inequality (\ref{energyeqint}), we have that
$$
\begin{array}{c}
 \jnt_0^1(\nu|\grad{\u}_n(s)|^2_2 +
\lambda\gamma|\grad z_n(s)|^2_2)\;ds 
=\jnt_{t_n}^{t_{n}+1}(\nu|\grad{\u}(t)|^2_2 +
\lambda\gamma|\grad z(t)|^2_2)\;dt
\\\qquad
 \leq \overline{\mathcal{E}}(\u(t_n),\p(t_n))-
 \overline{\mathcal{E}}(\u(t_n+1),\p(t_n+1))\longrightarrow 0\quad \mbox{ as }n\rightarrow \infty,
\end{array} 
$$
hence, 
$\grad \u_n\rightarrow 0\mbox{ strongly in }L^2(0,1;\L^2) $ and $\grad z_n\rightarrow 0\mbox{ strongly in }L^2(0,1;L^2)$. In particular, by using Poincar\'e inequality, one has
\begin{equation}\label{convz}
\u_n\rightarrow 0\:\mbox{ strongly in }L^2(0,1;\V)
\qquad\mbox{and}\qquad z_n\rightarrow 0\:\mbox{ strongly in }L^2(0,1;H^1_*). 
 \end{equation}
 Moreover, since $\u_n$ and $\partial_t \u_n$ are bounded in $L^{\infty}(0,1;\H)$ and $L^{4/3}(0,1;\V ')$ respectively, then using  Lemma~\ref{simon}, $\u_n\rightarrow 0$ in  $C([0,1];\V')$. In particular, $\u(t_n)=\u_n(0)\rightarrow0$ in  $\V '$, hence  $\u_\infty=0$ (owing to (\ref{phiinf})). Consequently, the whole trajectory $\u(t)\to 0$ as $t\to +\infty$.

On the other hand,  from the large time regularity of $(\u,\p,z)$ given in  (\ref{defi}), $\u_n$ is bounded in $L^2(0,1;\H^1)$, $z_n$ in $L^2(0,1;H^1_*)$ and $\p_n$  in $ L^\infty(0,1;H^2_1)$.  Moreover, using (\ref{H3-estim}), one also has that $\p_n$  is bounded in $L^2(0,1;H^{3}_1) $ and  from the $\p$-equation (\ref{P3s}) $\partial_t\p_n$ is bounded in $L^2(0,1;H^{-1}_*) $. Therefore, owing to Lemma~\ref{simon}, there exists a subsequence of $\p_n$ (equally denoted) and a limit function $\overline{\p}$ such that
\begin{equation}\label{convz-a}
 \p_n\rightarrow \overline{\p}\:\mbox{ strongly in $C^0([0,1]\times \overline\Om)\cap L^2(0,1;H^{2}_*)$ and weakly in $L^2(0,1;H^3_1)$.} 
  \end{equation}
In particular, $\p(t_n)=\p_n(0)\rightarrow \overline{\p}(0)$ in $C^0(\overline\Om)$, hence  $\overline{\p}(0)=\p_\infty$ (owing to (\ref{phiinf})). 
On the other hand,  $\partial_t \p_n$ converges weakly to $\partial_t \overline\p$ in $L^2(0,1;H^{-1}_*)$, hence taking limits in the variational formulation $ (\partial_t \psi_n,\widetilde z)+(\u_n\cdot\grad\psi_n,\widetilde z) +\gamma ( \nabla z_n,\nabla\widetilde z)=0$ 
for all $\widetilde z\in H^1_*$ and using convergences  (\ref{convz})-(\ref{convz-a}), we have that 
$\partial_t\p_n\rightarrow 0$ in $L^2(0,1;H^{-1}_*)$. Therefore, 
 $\partial_t \overline\p= 0$ and $\overline\p(t)$ is a constant function of $H^3_1$ for all $t \in [0,1]$, hence since $\overline{\p}(0)=\p_\infty$, one has  
\begin{equation}\label{barinfty}
\overline\p(t)=\p_\infty\in H^3_1\quad\mbox{ for all } t\in [0,1].
\end{equation}
Finally, by using  convergences  (\ref{convz})-(\ref{convz-a}) one also has arguing as in  (\ref{conv-G}) that $\overline{G}(\p_n)\rightarrow \overline{G}(\overline\p)$ weakly in $L^2(0,1;L^2)$, hence  taking limit as $n\rightarrow +\infty$ in the variational formulation  
$(z_n, \widetilde{\p})=-\eps(\nabla\Delta\p_n,\nabla\widetilde{\p})+(\overline{G}(\p_n),\widetilde{\p})$ for all $\widetilde{\p}\in H^1_*$, 
 we deduce 
$$-\eps(\nabla\Delta \overline\p,\nabla\widetilde{\p})+(\overline{G}(\overline\p),\widetilde{\p})=0, \quad \forall\,\widetilde{\p}\in H^1_*, \mbox{ a.e. }t\in (0,1).$$ 
Then, from (\ref{barinfty}),  $\p_\infty\in H^3_1$ and $-\eps(\nabla\Delta\p_\infty,\nabla\widetilde{\p})+(\overline{G}(\p_\infty),\widetilde{\p})=0$ for all $\widetilde{\p}\in H^1_*$ 
%
and the proof is finished. 
\end{proof}
\begin{theorem} \label{the:second}
Assume that $\widetilde{\mathcal{E}} (t)$ belongs to the equivalence class of  $\mathcal{E}(\u(t),\p(t))$, that is,  $\widetilde{\mathcal{E}} (t)=\mathcal{E}(\u(t),\p(t))$ a.e.~$t\geq 0$. Under the hypotheses of  Theorem~\ref{the:first}, there exists a unique limit $\p_\infty\in {\cal S}$ such that 
 $\p(t)\rightarrow\p_\infty$ in $H^2$-weakly as $t\uparrow +\infty$, i.e.  $\omega(\u,\p)=\{(0,\p_\infty)\}$. 
\end{theorem}
\begin{proof}
Let $\p_\infty\in{\cal S} $ such that $(0,\p_\infty)\in \omega(\u,\p)$, i.e.~there exists $t_n\uparrow +\infty$ such that $\u(t_n)\rightarrow 0$ weakly in $\L^2$ and $\p(t_n)\rightarrow \p_\infty$ weakly in  $H^2_1$ (and strongly in $H^1_*$).
\medskip

Without loss of generality, it can be assumed that $\widetilde{\mathcal{E}} (t)> E_\infty$ for all $t> 0$, because otherwise, if it exists some $\widetilde{t}>0$ such that $\widetilde{\mathcal{E}} (\widetilde t)= E_\infty$, then the energy inequality $(\ref{energyeqint-b})$   implies
$$\widetilde{\mathcal{E}} (t)= E_\infty,\quad \vert\grad\u(t)\vert_2^2=0 \quad \hbox{and} \quad \vert \grad z(t)\vert_2^2=0, \quad \forall\, t\geq\widetilde{t}.
$$
 In particular, $\u(t)=0$ and $z(t)$ is constant  for all $t\geq\widetilde{t}$, and  by using the $z$-equation (\ref{weak-form-2s}), $\partial_t\p(t)=0$,  hence $\p(t)=\p_\infty$ for all $t\geq\widetilde{t}$. In this setting the convergence of the whole $\p$-trajectory towards $\p_\infty$ is trivial.

Therefore, we can assume that $\widetilde{\mathcal{E}}(t)> E_\infty$ for all $t\geq 0$. In this case, the proof will be  divided into three steps.
\medskip

\noindent{\bf Step 1:} {\sl Assuming that there exists $t_1>0$ such that 
$$\Vert \p(t)-\p_\infty\Vert_1\leq\beta_1
\quad\mbox{ and }\quad
\vert \overline{\mathcal{E}}_b(\p(t))-\overline{\mathcal{E}}_b({\p}_\infty)\vert\leq\beta_2$$
 for a.e.~$ t\geq t_1\geq 0$, where $\beta_1>0, \beta_2>0$ are the constants appearing in Lemma~\ref{le:L-S2} (of Lojasiewicz-Simon's type), then the following inequalities hold:
\begin{equation}\label{stabs}
\fra{d}{dt}\Big( (\overline{\mathcal{E}}(\u(t),\p(t)-E_\infty)^\theta\Big)+{C}\,{\theta}\:(\vert \grad\u(t)\vert_2+
\vert \grad z(t)\vert_2)\leq 0,\quad a.e.~t\in (t_1,\infty)
\end{equation} 
\begin{equation}\label{stab2s}
\jnt_{t_{1}}^{t_2}\Vert\partial_t\p\Vert_{{H_*^{-1}}}\leq\fra{C}{\theta}(\overline{\mathcal{E}}(\u(t_1),\p(t_1))-E_\infty)^\theta,\qquad \forall\, t_2\in (t_1,\infty),
\end{equation}
where $\theta\in (0,1/2]$ is the  constant appearing in Lemma~\ref{le:L-S2}.  }
\medskip


Since $E_\infty$ is constant, we can rewrite  the energy inequality (\ref{energy-eqs}) as 
$$  \fra{d}{dt} (\widetilde{\mathcal{E}}(t)-E_\infty) + C\left(  |\nabla{\u}(t)|^2_2  +
|\grad z(t)|^2_2  \right)\leq 0, \quad\mbox{ a.e.}~t\ge 0.$$ 
By taking into account that $ \displaystyle |\nabla{\u}(t)|^2_2  +
|\grad z(t)|^2_2\geq  \frac{1}{2}\left(|\nabla{\u}(t)|_2  +
|\grad z(t)|_2\right)^2$ and the inequality 
$\displaystyle  \frac{1}{2} (|\grad\u(t)|_2+|\grad z(t)|_2) \ge C(|\u(t)|_2+\Vert z(t)\Vert_{H_*^{-1}})$,
 we obtain
$$  \fra{d}{dt} (\widetilde{\mathcal{E}}(t)-E_\infty) + C ( |{\u(t)}|_2  +   \Vert z(t)\Vert_{H_*^{-1}} )\left(  |\nabla{\u(t)}|_2  +
| \grad z(t)|_2    \right)\leq 0, \quad\mbox{ a.e.}~t\ge 0.$$
By using this expression and the time derivative of the (strictly positive) function 
$$H(t):=(\widetilde{\mathcal{E}}(t)-E_\infty)^\theta >0,
$$ we obtain
\begin{equation}\label{estim-1s}
 \fra{dH(t)}{dt} + \theta(\widetilde{\mathcal{E}}(t)-E_\infty)^{\theta-1}C (  |{\u(t)}|_2  +    \Vert z(t)\Vert_{H_*^{-1}})(\left  |\nabla{\u(t)}|_2  +
\vert \grad z(t)\vert_{2}  \right)\leq0, 
\quad \hbox{a.e.~$t\ge 0$.}
\end{equation}
On the other hand, 
by taking into account that $\vert \mathcal{E}_k(\u(t))\vert=\fra{1}{2}\vert\u(t)\vert_2^2\:$ and that  $|\u(t)|_2\le K$, we have
\begin{equation}\label{Ek}
\vert \mathcal{E}_k(\u(t))\vert^{1-\theta}=\fra{1}{2^{1-\theta}}\vert\u(t)\vert_2^{2(1-\theta)}=
\fra{1}{2^{1-\theta}}\vert\u(t)\vert_2^{1-2\theta}\vert\u(t)\vert
\leq C\vert\u(t)\vert_2 \quad a.e.~t\ge 0.
\end{equation}
Moreover, the Lojasiewicz-Simon inequality (\ref{conclu}) holds a.e.~$t\ge  t_1$, that is
\begin{equation}\label{conclu-a}
\vert \overline{\mathcal{E}}_b(\p(t))- E_\infty \vert^{1-\theta}\leq C \Vert z(t)\Vert_{H_*^{-1}},  \quad\mbox{ a.e.}~t\ge  t_1 .
\end{equation}
Hence, (\ref{conclu-a}) and (\ref{Ek}) 
give
$$
\begin{array}{c}
( \overline{\mathcal{E}}(\u(t),\p(t))-E_\infty)^{1-\theta}\leq
\vert \mathcal{E}_k(\u(t)) \vert^{1-\theta}+\vert \overline{\mathcal{E}}_b(\p(t))- E_\infty \vert^{1-\theta}\\
\leq C (\vert \u(t)\vert_2+
  \Vert z(t)\Vert_{H_*^{-1}})^{} \quad \mbox{ a.e.}~t\ge  t_1.
\end{array}$$
Therefore,  
\begin{equation}\label{interm-L-S}
( \overline{\mathcal{E}}(\u(t),\p(t))-E_\infty)^{\theta-1}(\vert \u(t)\vert_2+  \Vert z(t)\Vert_{H_*^{-1}})^{} \geq
C  \quad \mbox{ a.e.}~t\ge  t_1.
  \end{equation}
 By applying (\ref{interm-L-S}) in (\ref{estim-1s}) and that  $\overline{\mathcal{E}}(\u(t),\p(t))=\widetilde{\mathcal{E}}(t)$ a.e.~$t$, one has 
$$ \fra{dH(t)}{dt} + \theta\, C (\vert\grad \u(t)\vert_2+
\vert \grad z(t)\vert_{2})\le 0, \quad \mbox{ a.e.}~t\ge  t_1$$
hence $(\ref{stabs})$ is proved. Notice  that  the hypothesis $\overline{\mathcal{E}}(\u(t),\p(t))=\widetilde{\mathcal{E}}(t)$ for almost every $t$ is the key point to arrive at $(\ref{stabs})$. In particular, this hypothesis implies that the integral and differential versions of the energy law (\ref{energyeqint}) and (\ref{energy-eqs}) are satisfied by $\overline{\mathcal{E}}(\u(t),\p(t))$ a.e.\ in time. In fact, energy law (\ref{energy-eqs}), changing $\widetilde{\mathcal{E}}(t)$ by  $\overline{\mathcal{E}}(\u(t),\p(t))$, is the crucial hypothesis imposed in Remark 2.4 of \cite{PRS}. 

  Fixed any $t_2\in (t_1,\infty)$, integrating $(\ref{stabs})$ into $[t_1, t_2]$, taking into account that $(\overline\mathcal{E}(\u(t_2),\p(t_2))-E_\infty)^\theta>0$, we have
\begin{equation}\label{estim-2s}
\theta\, C\jnt_{t_1}^{t_2}(\vert\grad \u(t)\vert_2+ \vert\grad z(t)\vert_{2}) dt
  \leq (\overline{\mathcal{E}}(\u(t_1),\p(t_1))-E_\infty)^\theta.
\end{equation}
From the equation  (\ref{weak-form-2s}), by using the weak regularity $\p\in L^\infty((0,+\infty)\times \Om)$, then 
$$\Vert\partial_t\p(t)\Vert_{H_*^{-1}}\leq C(\vert\u(t)\vert_2+\vert  \grad z(t)\vert_2)
\qquad a.e.~t\ge 0.
$$
By using this inequality in  (\ref{estim-2s}), then (\ref{stab2s}) is attained. 
 
 \medskip

\noindent{\bf Step 2:}  {\sl There exists a sufficiently large $n_0$ such that  $\Vert
\p(t)-\p_\infty\Vert_1\leq\beta_1$ and $\vert \overline{\mathcal{E}}_b(\p(t))-\overline{\mathcal{E}}_b({\p}_\infty)\vert\leq\beta_2$ for all $t\geq t_{n_0}$ ($\beta_1, \beta_2$  given in Lemma~\ref{le:L-S2})}.
\medskip

Since $\p(t_n)\rightarrow \p_\infty$ strongly  in $H^1_*$ and $\overline{\mathcal{E}}(\u(t_n),\p(t_n)) \searrow E_\infty=\overline{\mathcal{E}}_b(\p_\infty)$ in $\mathbb{R}$ (see (\ref{asenergys})), then  for any $\delta\in (0,\beta_1)$, there exists an integer $N(\delta)$ such that, for all $n\geq N(\delta)$,
\begin{equation}\label{epsis}
\Vert \p(t_n)-\p_\infty\Vert_1\leq\delta\quad\mbox{ and }\quad\fra{1}{\theta}(\overline{\mathcal{E}}_b(\p(t_n))- E_\infty)^\theta\leq\delta.
\end{equation}
For each $n\geq N(\delta)$, we define 
$$
\overline{t}_n:=\sup\{t: t>t_n, \,\Vert \p(s)-\p_\infty \Vert_1<\beta_1\quad\forall s\in [t_n,t)\}.$$ 
It suffices to prove that  $\overline{t}_{n_0}=+\infty$ for some  $n_0$. Assume by contradiction that  $t_n<\overline{t}_n<+\infty$ for all $n$, hence $\Vert \p(\overline{t}_n)-\p_\infty \Vert_1=\beta_1$ and $\Vert \p({t})-\p_\infty \Vert_1<\beta_1$ for all $t\in [t_n, \overline t_n)$.   By applying Step 1 for all $t\in [t_n,\overline{t}_n]$, from $(\ref{stab2s})$ and $(\ref{epsis})$ we obtain, 
$$
\jnt_{t_{n}}^{\overline{t}_n}\Vert\partial_t\p\Vert_{H_*^{-1}}\leq C\delta, \quad \forall\,n\geq N(\delta).
$$
Therefore, 
$$
\Vert \p(\overline{t}_n)-\p_\infty\Vert_{H_*^{-1}}\leq
\Vert \p(t_n)-\p_\infty\Vert_{H_*^{-1}}+\jnt_{t_{n}}^{\overline{t}_n}\Vert\partial_t\p\Vert_{H_*^{-1}}\leq (1+C)\delta,$$
which implies that $\lim_{n\rightarrow +\infty}\Vert \p(\overline{t}_n)-\p_\infty\Vert_{H_*^{-1}}=0$. 

On the other hand, we will prove that $ \p(\overline{t}_n)$ is bounded in $H^2_1$. Indeed, from (\ref{asenergys}), $\overline{\mathcal{E}} (\u(\overline{t}_n),\p(\overline{t}_n))$  is bounded in $\mathbb{R}$, therefore in particular 
\begin{equation}\label{mu-estim}
\mu(\p(\overline{t}_n)+m_0)=-\eps\Delta \p(\overline{t}_n)+\frac{1}{\eps}F'(\p(\overline{t}_n)+m_0)\quad \hbox{is bounded in $L^2(\Om)$.}
\end{equation}
But, since $ F'(\p(\overline{t}_n)+m_0)$ is bounded in $L^\infty(\Om) $ (because $ \p\in C([0,+\infty)\times \overline \Om)$), then $\Delta \p(\overline{t}_n)$ is bounded in $L^2(\Om)$. Therefore using the $H^2$-regularity of problem $(P_0)$ (see (\ref{n4w})), one has that $ \p(\overline{t}_n)$ is bounded in $H^2_1$.
  
  Consequently, $ \p(\overline{t}_n)$    is relatively compact in $H^1_*$, hence there exists a subsequence of $ \p(\overline{t}_n)$, which is still denoted as $ \p(\overline{t}_n)$, that converges to $\p_\infty$ in $H^1_*$-strong. Hence  $\Vert \p(\overline{t}_n)-\p_\infty\Vert_{1}<\beta_1$ for a sufficiently large $n$, which contradicts the definition of $\overline{t}_n$. 
\medskip

\noindent{\bf Step 3:} {\sl There exists a unique $\p_\infty$ such that $\p(t) \to \p_\infty$ weakly in $H^2_1$ as $t\uparrow +\infty$.}
\medskip

By using Steps 1 and 2,  (\ref{stab2s})  can by applied, for all $t_1,t_0: t_1>t_0\geq t_{n_0}$, hence
$$
\Vert \p(t_1)-\p(t_0)\Vert_{H_*^{-1}}\leq
\jnt_{t_0}^{t_1}\Vert\partial_t\p\Vert_{H_*^{-1}} \to 0,\quad \hbox{as $t_0,t_1\to +\infty$.}$$
Therefore, 
$(\p(t))_{t\geq {t_{n_0}}}$ is a Cauchy sequence in $H_*^{-1}$ as $t\uparrow +\infty$, hence 
there exists a unique $\p_\infty\in H_*^{-1}$ such that $\p(t)\to\p_\infty$ in $H_*^{-1}$ as $t\uparrow +\infty$. Finally, the   convergence in $H^2_1$-weak by sequences of $\p(t)$ proved in Theorem~\ref{the:first}, yields  to $\p(t)\to \p_\infty$ in $H^2_1$-weak, and the proof is finished.
\end{proof}
\section{Higher regularity  for the phase variable}\label{HS}
We consider the following zero-mean regular space:
$$ 
H^4_2=\left\{w\in H^4\cap L^2_* ; \:\partial_n w\vert_{\partial\Om}=0,\:\partial_n \Delta w \vert_{\partial\Om}=0\right\}.
$$
From now on, we assume $\Om$ sufficiently regular (for instance $\partial\Om\in C^4$) such that  the $H^4$-regularity of the Poisson-Neumann problem ${\rm (P_0)}$ holds.

For any $w \in H^4_2$,  taking $v=\Delta w$ in ${\rm (P_0)}$, since $\partial_n(\Delta\p)|_{\partial\Om}=0$ and $\int_\Om\Delta\p=0$, then the $H^2$-regularity of ${\rm (P_0)}$ implies 
\begin{equation} \label{n5wA}
\Vert \Delta w \Vert_2\leq C  \vert\Delta^2 w \vert_2  \quad\forall\, w \in H^4_2
\end{equation}
 In particular,  
from (\ref{n5wA}) and the $H^4$-regularity of ${\rm (P_0)}$,  one has
\begin{equation}\label{n4p}
\begin{array}{c}
\Vert w \Vert_4\leq C  \vert\Delta^2 w \vert_2  \quad\forall\, w \in H^4_2
\end{array}
\end{equation}
\subsection{Global in time strong regularity for $\p$}
We will see that $\p\in  L^\infty(0,+\infty;H^3_1)$ if the data $\psi_0\in H^3_1$. For this, we argue in a formal manner, that could be rigorously justified via the Galerkin method. 
 
We add the $z$-equation (\ref{weak-form-2s}) tested by $\partial_t \p$ $(\in H^1_*)$, and the  $\p$-equation (\ref{weak-form-3s})  by $\gamma\,\Delta\partial_t \p$ ($\in H^2_1$ owing to $\grad \partial_t\Delta\p\cdot\n\vert_{\partial\Om}=0 $). If we integrate twice by parts in the term $\gamma\,(\overline{G}(\p)-z,\Delta\partial_t\p)$ of (\ref{weak-form-3s}), taking into account that $\grad \partial_t\p\cdot\n\vert_{\partial\Om}=0$, $\grad \overline{G}(\p)\cdot\n\vert_{\partial\Om}=0$ and $\grad z\cdot\n\vert_{\partial\Om}=0$,  one has 
$$
\gamma\,(\overline{G}(\p)-z,\Delta\partial_t\p)
=\gamma\,(\Delta \overline{G}(\p)-\Delta z,\partial_t\p).
$$
 Therefore, the term  $\gamma\,(\Delta z, \partial_t\p)$ cancels, remaining:
$$
  \fra{\eps}{2}\fra{d}{d t}\vert \grad \Delta\p \vert^2_2+\vert\partial_t\p\vert^2_2=-(\u\cdot\grad\p,\partial_t \p)-
\gamma\,(\Delta \overline{G}(\p),\partial_t\p) ,
$$
hence applying Holder and Young inequalities
\begin{equation}\label{str}
  \eps \fra{d}{d t}\vert \grad \Delta\p \vert^2_2+\vert\partial_t\p\vert^2_2
 \leq C(\vert\u\cdot\grad\p\vert_2^2+\vert\Delta \overline{G}(\p)\vert_2^2).
\end{equation}
From the $\p$-equation (\ref{P4s}),  by using  (\ref{L})  and (\ref{n4p}), we obtain 
\begin{equation}\label{phiH4}
\Vert\p\Vert_4 \leq C\, |\Delta^2 \p |_2 \leq C(\vert \overline{G}(\p)\vert_2+\vert z\vert_2)\leq C(1+\vert z\vert_2).
\end{equation}
In particular, since $\p\in L^\infty(0,+\infty;H^2)$,
\begin{equation}\label{DLbis0}
\vert \Delta \overline{G}(\p)\vert_2 \leq C(1+\Vert\p\Vert_4) \leq C(1+\vert  z\vert_2).
\end{equation}

By using  (\ref{phiH4}) and (\ref{DLbis0})  in (\ref{str}), and taking into account that $\Vert \p\Vert_3$ is equivalent to $\vert \grad\Delta\p\vert_2$ (see (\ref{n4wA})), one has  
\begin{equation}\label{trbi}
\fra{d}{d t}\Vert \p \Vert^2_3+C_0(\Vert\p\Vert^2_4+
\vert\partial_t\p\vert_2^2)\leq     
C\left( 1+ \Vert \u \Vert_1^2 +\vert  z\vert_2^2\right) .
\end{equation}
By denoting
$$\Phi(t):= \Vert \p \Vert^2_3,
\qquad
B(t):=\Vert \u \Vert_1^2 +\vert z\vert_2^2,$$
then (\ref{trbi}) yields to  
\begin{equation}\label{dine}
\Phi' (t)+C_0\Phi (t) \leq C(1+B (t) ).
\end{equation}
Multiplying (\ref{dine}) by $e^{C_0t}$ and integrating in time, we obtain 
$$\Phi(t)\leq \Phi(0)e^{-C_0t}+Ce^{-C_0t}\int_0^te^{C_0s}(1+B(s))\,ds.$$
In particular, 
$$\Phi(t)\leq \Phi(0)+C(1-e^{-C_0t})+C\jnt_0^tB(s)\,ds.$$
 Since $B(t)\in L^1(0,+\infty)$, we have that $\Phi\in L^\infty(0,+\infty)$, that means
 \begin{equation}\label{estim4-1}
 \p  \in L^\infty(0,+\infty;H^3_1).
 \end{equation} 
  Moreover, integrating in time in (\ref{trbi}), we obtain 
\begin{equation}\label{estim4-2}
 \p  \in L^2_{loc}([0,+\infty);H^4_2) \quad \hbox{and} \quad \partial_t \p  \in L_{loc}^2(0,+\infty;L^2_*).
\end{equation} 
 In particular, using this improved regularity, the phase equations (\ref{P3s})-(\ref{P4s}) are satisfied point-wisely a.e.~$t\in (0,+\infty)$, if the data $\psi_0\in H^3_1$.
\begin{remark}
It has been possible to obtain higher estimates for the phase variable without improving estimates for the velocity and pressure.
\end{remark}
\subsection{Improving the convergence of the phase trajectory}
The previous extra global in time regularity obtained for the phase variable $\p$ allows to obtain a convergence of the whole trajectory of $\p(t)$ in a more regular space, replacing the convergence in $H^2$ by $H^3$. Indeed, 
fixed an initial data $(\u_0,\p_0)\in \H\times H^3_1$, we replace 
the $\omega$-limit and the equilibrium point sets defined in (\ref{def-w}) and (\ref{def-S}) by the following more regular sets:
$$
\begin{array}{r}
\omega_{reg}(\u,\p)=\{(\u_\infty,\p_\infty)\in \H\times H^3_1: \exists \{t_n\}\uparrow+\infty \mbox{ s.t.\ }\qquad\qquad\qquad\\({\u}(t_n),{\p}(t_n))\rightarrow(\u_\infty,\p_\infty) \mbox{ weakly in } \H\times H^3_1\},
 \end{array}
$$
 $$
{\cal S}_{reg}=\{\p\in H^4_2(\Omega)\ : \ \eps \Delta^2\p+\overline{G}(\p)=0\hbox{ a.e in }\Om\}.
$$

In this setting, theorems \ref{the:first} and \ref{the:second} imply in particular, the following statements:
\begin{theorem} \label{the:first-semif}
The set $\omega_{reg}(\u,\p)$  is  nonempty  and $\omega_{reg}(\u,\p)\subset \{0\}\times  {\cal S}_{reg}$. Moreover, for any $\p_\infty \in {\cal S}_{reg}$ with  $ (0,\p_\infty)\in \omega_{reg}(\u,\p)$, then $\overline{\mathcal{E}}_b(\p_\infty)=E_\infty$. 
\end{theorem}
\begin{theorem} \label{the:second-semif}
There exists a unique    $ \p_\infty \in {\cal S}_{reg}$ such that  
 $\p(t)\rightarrow\p_\infty$ in $H^3_1$ weakly as $t\uparrow +\infty$, i.e.  $\omega_{reg}(\u,\p)=\{(0,\p_\infty)\}$. 
\end{theorem}
\begin{remark} In the previous work \cite{cedya-15}, we have proved the convergence of the whole trajectory $(\u(t),\p(t))\rightarrow (0,\p_\infty)$ weakly in  $\V\times H^3_1$,  by using different arguments that now.  In \cite{cedya-15},  by applying a more standard Lojasiewicz-Simon inequality (given in Lemma 5.2 of \cite{WuXu}) jointly to the extra regularity (\ref{estim4-1}) and (\ref{estim4-2}), one has the existence and uniqueness of strong solutions (in velocity-pressure and phase) for sufficiently large times and the continuous dependence of local in time strong solutions. Now, we have proved the same type of results but using only the global weak regularity of the problem.
\end{remark}
\subsection{Additional $H^6$-regularity for the phase}\label{H6}
It will be possible to obtain the following  higher regularity in space for the phase variable $\p  \in L^2_{loc}(0,+\infty;H^6_2)$, by imposing only more regularity of the domain. For this, we consider the following zero-mean regular space:
$$ 
H^6_2=\left\{w\in H^6\cap L^2_* ; \:\partial_n w\vert_{\partial\Om}=0,\:\partial_n \Delta w \vert_{\partial\Om}=0\right\}
$$
and the second order elliptic  (nonhomogeneous) Neumann-Poisson problem: 
$$
\left\{
\begin{array}
{l}
-\Delta v=f\quad\mbox{in }\Omega\\
\partial_nv\vert_{\partial\Omega}=g,\quad 
 \jnt_\Om v\;d\x=0,
\end{array}
\right.
\leqno{\rm (P_g)}
$$
where $f\in L^2(\Om)$ and $g\in H^{1/2}(\partial\Om)$, satisfying the compatibility  condition $$\displaystyle\int_\Omega f\, d\x+ \int_\Omega g\, d\x=0.$$
We will use problem ${\rm (P_g)}$ for $v=\Delta^2\psi$, therefore we need to obtain an expression of $\partial_n(\Delta^2\psi)|_{\partial\Om}$.

From (\ref{P4s}),  $\Delta^2\p=\fra{1}{\eps}(z-\overline{G}(\p))$. Since $0=\partial_n\p|_{\partial\Om}=\partial_n \Delta \p|_{\partial\Om}=\partial_n{\mu}(\p)|_{\partial\Om}=\partial_n z|_{\partial\Om}$, then
$$ 
\begin{array}{l}
\partial_n(\Delta^2\psi)|_{\partial\Om}= -\fra{1}{\eps}\partial_n(\overline{G}(\p))|_{\partial\Om}= \fra{1}{\eps^2}\partial_n(\Delta F'(\p+m_0))|_{\partial\Om}\\
\phantom{\partial_n(\Delta^2\psi)|_{\partial\Om}}
= \fra{1}{\eps^2}\Big( 
F^{iv)}(\p+m_0)\vert\grad\p\vert^2\partial_n\p
+2 F'''(\p+m_0)\grad\p\cdot\partial_n(\grad\p)
 \\
\phantom{\partial_n(\Delta^2\psi)|_{\partial\Om}}
\quad +
F'''(\p+m_0)\Delta\p\,\partial_n\p +
F''(\p+m_0)\partial_n(\Delta\p)\Big) |_{\partial\Om}
\\
\phantom{\partial_n(\Delta^2\psi)|_{\partial\Om}}
= \fra{2}{\eps^2}
F'''(\p+m_0)\grad\p\cdot\partial_n(\grad\p) |_{\partial\Om}
\end{array}
$$

From now on, we assume $\Om$ sufficiently regular (for instance $\partial\Om\in C^6$) such that  the $H^6$-regularity of $
{\rm (P_0)}$ and the $H^2$-regularity of ${\rm (P_g)}$ hold. From the $H^6$ and the $H^4$-regularity of ${\rm (P_0)}$,
\begin{equation}\label{n6pa2}
\begin{array}{c}
\Vert w \Vert_6\leq C  
\Vert\Delta w \Vert_4\leq C  \Vert\Delta^2 w \Vert_2 \quad \forall\, w \in H^6_2,
\end{array}
\end{equation}
and from the $H^2$-regularity of ${\rm (P_g)}$ for $v=\Delta^2\p$ and $g=\partial_n(\Delta^2\psi)|_{\partial\Om}$,
\begin{equation}\label{n6pb2}
\Vert\Delta^2\p\Vert_2
\leq C( \vert\Delta^3\p\vert_2+\Vert g\Vert_{1/2,\partial\Om}).
\end{equation}
From (\ref{P3s}) and  (\ref{P4s}), and using (\ref{DLbis0}), we have that
\begin{equation}\label{delta3}
\vert\Delta^3\p\vert_2\leq C(\vert\partial_t\p\vert_2+\vert\u\cdot\grad\p\vert_2
+ \vert \Delta \overline{G}(\p)\vert_2)\leq C(\vert\partial_t\p\vert_2+\Vert\u\Vert_1+1+\vert z\vert_2).
\end{equation}
On the other hand, by using $\Vert \p\Vert_{2}\leq C$, (\ref{n6pa2}), and (\ref{n6pb2}) for $g=
\fra{2}{\eps^2}
F'''(\p+m_0)\grad\p\cdot\partial_n(\grad\p)
$, 
we obtain
$$\begin{array}{l}
\Vert\p\Vert_6
\leq C( \vert\Delta^3\p\vert_2+\Vert F'''(\p+m_0) \partial_j  \p \,\partial^2_{ij} \p \, n_i\Vert_{1/2,\partial\Om})
\\ \phantom{\Vert\Delta^2\p\Vert_2}
\leq C( \vert\Delta^3\p\vert_2+\Vert \p\Vert_{3}^2)
\leq C( \vert\Delta^3\p\vert_2+\Vert \p\Vert_{2}\Vert \p\Vert_{4}).
\end{array}$$
Taking into account the interpolation inequality  
$\Vert \p\Vert_{4}\leq \Vert \p\Vert_{2}^{1/2}\Vert \p\Vert_{6}^{1/2}$, we obtain that
$$ \Vert\p\Vert_6
\leq C( \vert\Delta^3\p\vert_2+\Vert \p\Vert_{6}^{1/2}),
$$
hence 
$$
\Vert\p\Vert_6 \le C(1+\vert\Delta^3\p\vert_2) .
$$
 Therefore,  taking into account (\ref{delta3})
\begin{equation}\label{6}
\Vert\p\Vert_6^2\leq C(1+\vert\partial_t\p\vert_2^2+\Vert\u\Vert_1^2 + \vert z\vert_2^2).
\end{equation}
By using  (\ref{6}) in (\ref{trbi}), we obtain
\begin{equation}\label{trbis}
\fra{d}{d t}\Vert \p \Vert^2_3+C_0(\Vert\p\Vert^2_6+
\vert\partial_t\p\vert_2^2)\leq     
C\left( 1+ \Vert \u \Vert_1^2 +\vert  z\vert_2^2\right) .
\end{equation}
By integrating in time in (\ref{trbis}), since the initial phase $\psi_0\in H^3_1$, we obtain 
$$\p  \in L^2_{loc}(0,+\infty;H^6_2).$$
 Finally, from (\ref{P4s}), $$z\in L^2_{loc}(0,+\infty;H^2_1).$$
 In particular, using these improved estimates and  the phase equations (\ref{P3s})-(\ref{P4s}), we have that the following sixth order equation
 $$
  \partial_t \psi+\u\cdot\grad\psi -\gamma \eps \Delta^3 \p - \gamma 
  \Delta(\overline{G}(\p))=0,
 $$
 is satisfied point-wisely a.e.~$t\in (0,+\infty)$.

\section{Conclusions and Perspectives}
For the Navier-Stokes-Cahn-Hilliard model introduced in this paper we have proved the convergence of the  trajectory of any global weak solution to a single equilibrium point. Moreover, the regularity for the phase is improved  without the need of more regularity for the velocity and pressure variables, and therefore it is not necessary to impose large time or large viscosity constraints. 

Bearing in mind the results obtained in this paper,  it seems achievable  to obtain (rational) convergence rate estimates of the convergence of trajectories in a similar way to \cite{GalGrasselli}. Finally, it would be interesting to study if local minimizers of the elastic bending energy are stable, as was proved in \cite{WuXu} for a Navier-Stokes-Allen-Cahn problem modeling vesicle membranes.

\vspace{0.5cm}
\noindent
{\bf Acknowledgements:}

This work has been partially
financed by the MINECO grant MTM2015-69875-P (Spain) with the participation of FEDER.


\begin{thebibliography}{99}
\small

\bibitem{Brezis} H.\ Br\'ezis, \emph{Analyse fonctionnelle : Theorie et applications.} Dunod (1999) 

\bibitem{CampHdezMach}  F.\ Campelo, A.\ Hern\'andez-Machado, \emph{Model for Curvature-Driven Pearling Instability in Membranes,} Phys. Rev. Lett. 99(8) (2007) 1-4.

\bibitem{smectic} B.\ Climent-Ezquerra, F.\ Guill\'en-Gonz\'alez. \emph{Global in time solution and time-periodicity for a Smectic-A Liquid Crystal Model,} Commun. Pure Appl. Anal. 9, no. 6, (2010), 1473-1493.

\bibitem{nu} B.\ Climent-Ezquerra, F.\ Guill\'en-Gonz\'alez. \emph{Convergence to equilibrium for smectic-A liquid crystals in 3D domains without constraints for the viscosity,} Nonlinear Anal. 102 (2014) 208-219.

\bibitem{cedya-15} B.\ Climent-Ezquerra, F.\ Guill\'en-Gonz\'alez. \emph{Long-time behavior of a Cahn-Hilliard-Navier-Stokes vesicle-fluid interaction model,}  Trends in differential equations and applications, 125-145, SEMA SIMAI Springer Ser., 8, Springer, 2016.

\bibitem{RBGG} B.\ Climent-Ezquerra, F.\ Guill\'en-Gonz\'alez,  M.A.\ Rodr\'{\i}guez-Bellido. \emph{Stability for Nematic Liquid Crystals with Stretching Terms,}
International Journal of Bifurcations and Chaos,  20, (2010), 2937-2942.

\bibitem{Colli-Laurencot} P.\ Colli, P. Laurencot. \emph{A phase-field approximation of the Willmore flow with volume and area cosntraints}. 
SIAM J. Math. Anal., 44(6), (2012), 3734-3754

\bibitem{ConstantinFoias} P.\ Constantin, C.\ Foias.  \emph{Navier-Stokes Equations}
University of Chicago Press (1988).

\bibitem{DuLiLiu} Q.\ Du, M.\ Li, C.\ Liu,  \emph{Analysis of a phase field Navier-Stokes vesicle-fluid interaction model},
 Disc. Cont. Dyn. Sys. B 8, no. 3, (2007),  539--556.
 
 \bibitem{DuLiuWang}  Q.\ Du, C.\ Liu, X.\ Wang,  \emph{A phase field approach in the numerical study of the elastic
bending energy for vesicle membranes}, Journal of Computational Physics, 198 (2004), 450-468.

\bibitem{ErnGuermond}
A.\ Ern and J.-L.\ Guermond,
\emph{ Theory and Practice of Finite Elements}, Series of Applied
Mathematical Sciences v. 159, Springer-Verlag, 2004

\bibitem{GalGrasselli} C.G.\ Gal, M.\ Grasselli, \emph{Asymptotic behavior of a Cahn-Hilliard-Navier-Stokes system in 2D}, Ann. I. H. Poincaré (C) Non Linear Analysis 27 (2010) 401-436

\bibitem{Helfrich} W.\ Helfrich, \emph{Elastic properties of lipid bilayers-theory and possible experiments}, Z.\ Naturforsch. C 28, 693-703 (1973).

\bibitem{Huang} S.Z.\ Huang. \emph{Gradient Inequalities: with Applications to Asymptotic Behavior and Stability of Gradient-like Systems, } Mathematical Surveys and Monographs, vol. 126 AMS (2006)

\bibitem{LiuTakahashiTucsnak}   C.\ Liu, T.\ Takahashi,  M.\ Tucsnak,  \emph{Strong solutions for a phase-filed Navier-Stokes Vesicle-Fluid interaction model}, J. Math. Fluid Mech., 14 (2012), 177-195.  

\bibitem{PRS} H.\ Petzeltova, E.\ Rocca, G.\ Schimperna, 
\emph{On the long-time behavior of some mathematical models for nematic liquid crystals}, 
Calc. Var., 46 (2013), 623-639.

\bibitem{PonceTiti} G.\ Ponce, R.\ Racke, T.S.\ Sideris, E.S.\ Titi
\emph{Global Stability of Large Solutions to the 3D Navier-Stokes Equations}, 
Commun. Math. Phys. (1994), 329-341.

\bibitem{Royden} H.L.\ Royden, \emph{Real Analysis.} Macmillan (1988) 

\bibitem{SegattiWu} A.\ Segatti, H.\ Wu. \emph{Finite dimensional reduction and convergence to equilibrium for incompressible Smectic-A liquid crystal flows},  SIAM J. Math. Anal., 43(6) (2011), 2445-2481.

\bibitem{Simon} J.\ Simon \emph{Compact sets in the space $L^p(0,T;B)$} Annali di Matematica Pura e Applicata, 146(4) (1987), 65-96.

\bibitem{Temam} R.\ Temam, \emph{Navier-Stokes equations: theory and numerical analysis},  North-Holland (1977).

\bibitem{WuXu} H.\ Wu, X.\ Xu,  \emph{Strong solutions, global regularity and stability of a hydrodynamic system modeling vesicle and fluid interactions}, SIAM J. Math. Anal., 45 (1) (2013), 181-214.


\end{thebibliography}
\end{document}